\documentclass[ijoc,sglanonrev]{informs4}
\usepackage{eqndefns-left} 
\RequirePackage{tgtermes}
\RequirePackage{newtxtext}
\RequirePackage{newtxmath}
\RequirePackage{bm}
\RequirePackage{endnotes}

\usepackage{amsmath,amssymb,amsfonts}

\usepackage{hyperref}

\usepackage{url}
\usepackage{lineno}

\OneAndAHalfSpacedXII 

\usepackage{algorithm}
\usepackage{algpseudocode}
\usepackage{tikz}

\usepackage{natbib}
\bibpunct[, ]{(}{)}{,}{a}{}{,}%
\def\bibfont{\small}%

\usepackage{rotating}
\usepackage{fancyvrb}
\usepackage{tcolorbox}
\usepackage{bbm}

\usepackage{listings}

\usepackage{booktabs}
\usepackage{multirow}
\usepackage{graphicx}

\makeatletter
\newcommand{\AppendixHyperFix}{%
	\renewcommand{\theHsection}{appendix.\Alph{section}}%
	\renewcommand{\theHsubsection}{appendix.\Alph{section}.\arabic{subsection}}%
	\renewcommand{\theHsubsubsection}{appendix.\Alph{section}.\arabic{subsection}.\arabic{subsubsection}}%
	\renewcommand{\theHequation}{appendix.\Alph{section}.\arabic{equation}}%
}
\makeatother

\makeatletter
\newenvironment{problemequation}[2]{%
	\begingroup
	\renewcommand{\theequation}{#1}%
	\@ifundefined{theHequation}{}{%
		\renewcommand{\theHequation}{#2}%
	}%
	\begin{equation}%
	}{%
	\end{equation}%
	\addtocounter{equation}{-1}%
	\endgroup
}
\makeatother

\EquationsNumberedThrough    

\TheoremsNumberedThrough     
\ECRepeatTheorems  %

\MANUSCRIPTNO{IJOC-0001-2024.00}

\usepackage{mathtools}

\usepackage[normalem]{ulem} 

\newcommand{\cB}{{\mathcal{B}}}

\newcommand{\cD}{{\mathcal{D}}}

\newcommand{\cJ}{{\mathcal{J}}}

\newcommand{\cN}{{\mathcal{N}}}

\newcommand{\cR}{{\mathcal{R}}}

\newcommand{\cX}{{\mathcal{X}}}
\newcommand{\cY}{{\mathcal{Y}}}

\newcommand{\RR}{\mathbb{R}} 

\newcommand{\dist}{\mathrm{dist}}    
\newcommand{\Diag}{{\mathrm{Diag}}} 

\newcommand{\bc}{\begin{center}}
\newcommand{\ec}{\end{center}}

\newcommand{\bdm}{\begin{displaymath}}
\newcommand{\edm}{\end{displaymath}}

\newcommand{\beq}{\begin{equation}}
\newcommand{\eeq}{\end{equation}}

\newcommand{\bfl}{\begin{flushleft}}
\newcommand{\efl}{\end{flushleft}}

\newcommand{\bt}{\begin{tabbing}}
\newcommand{\et}{\end{tabbing}}

\newcommand{\beqn}{\begin{eqnarray}}
\newcommand{\eeqn}{\end{eqnarray}}

\newcommand{\beqs}{\begin{align*}} 
\newcommand{\eeqs}{\end{align*}}  

\begin{document}
	
	
	\RUNAUTHOR{Luo, Liu, Liu, and Chen}
	
	\RUNTITLE{{Sequential smoothing majorant  stochastic approximation method}  for nonconvex--nonconcave minimax problems}
	
	\TITLE{{A sequential smoothing majorant  stochastic approximation method}  for nonconvex--nonconcave minimax problems}
	
	\ARTICLEAUTHORS{%
			\AUTHOR{Zhouxing Luo}
			\AFF{ Institute of Computational Mathematics and Scientific/Engineering Computing \\
				Academy of Mathematics and Systems Science, Chinese Academy of Sciences \\
				Beijing,  \EMAIL{luozhouxing18@mails.ucas.ac.cn }}
			
			\AUTHOR{Wei Liu}
			\AFF{Department of Applied Mathematics\\
				The Hong Kong Polytechnic University\\
				Hong Kong, \EMAIL{lwdsdqqb@gmail.com}}
			
			\AUTHOR{Xin Liu}
			\AFF{Institute of Computational Mathematics and Scientific/Engineering Computing \\
				Academy of Mathematics and Systems Science, Chinese Academy of Sciences \\
				Beijing, \EMAIL{liuxin@lsec.cc.ac.cn}}
			
			\AUTHOR{Xiaojun Chen}
			\AFF{Department of Applied Mathematics\\
				The Hong Kong Polytechnic University\\
				Hong Kong, \EMAIL{xiaojun.chen@polyu.edu.hk}}
		} 
		
		\ABSTRACT{%
			We propose a  sequential smoothing majorant  stochastic approximation (SMSA) method for  nonconvex--nonconcave minimax optimization problems. To overcome the nonconcavity of the inner maximization problem, we introduce a new majorant stochastic approximation with an $O(\beta_N^2)$ accuracy bound,  where $\beta_N$ is the sample coverage radius. The accuracy bound significantly improves upon the $O(\beta_N)$ approximation bound for the standard stochastic approximation. Moreover, we establish nonasymptotic bounds for both global optimal values and minimizer sets, and prove consistency for the
			Clarke stationary points, as $\beta_N\downarrow  0$ almost surely.  
			We show that the generated sequence by the SMSA method is bounded, and that the returned point is an approximate Clarke stationary point of the   majorant stochastic 	approximation models. 
			Numerical experiments on a synthetic toy example and robust logistic regression on two UCI datasets demonstrate improved approximation fidelity and lower mean robust test losses relative to the standard sampled approximation. 
		}%
		
		\FUNDING{This research was supported by [grant number, funding agency].}
		
		
		
		
		\SUBJECTCLASS{MSC2020: 90C30, 90C06, 49M37}
		
		
		\maketitle
		

		\section{Introduction}\label{sec:intro}

		In this paper, we consider the following nonconvex--nonconcave minimax problem:
			\begin{equation}
				\label{original_minimax_problem}
				\min_{x \in \cX} \; \max_{y \in \cY} \; f(x,y),
			\end{equation}
			where $\cX\subseteq\mathbb{R}^n$ and $\cY\subseteq\mathbb{R}^m$ are nonempty, compact, and convex sets, and the function \mbox{$f:\cX\times \cY \to \mathbb{R}$} is  continuously 
			differentiable, but 
			not necessarily convex in $x$ nor concave in $y$.
			Problems of this form arise in robust learning, adversarial classification, uncertainty-aware estimation, resilient planning, stochastic control, reservoir management, and data-driven logistics, among others; in all of these settings, the order of decision-making is essential as \(x\) is chosen first while \(y\) represents an adverse or uncertain response that emerges subsequently \citep{Pfetsch2023ResilientSystems,Noyan2022DecisionDependentDRO,Shehadeh2023MobileFacilityDRO,goodfellow2014explaining,huang2015learning,madry2018towards,kuhn2024distributionallyrobustoptimization,rahimian2019distributionally}. 
			We briefly introduce a representative example.
			
			\begin{example}[Adversarial training~\citep{madry2018towards}]
				Adversarial training is one of the most effective approaches for defending deep learning models against adversarial examples. Given training data \(\{(\zeta_j,\eta_j)\}_{j=1}^N\), a shared-perturbation formulation of adversarial training is
				\[
				\min_{x\in\cX}\max_{y\in\cY}
				\frac{1}{N}\sum_{j=1}^N \ell(x;\zeta_j+y,\eta_j)+\varphi(x),
				\]
				where \(x\) denotes the model parameters, \(\ell(\cdot;\zeta,\eta):\cX\to\RR\) is a differentiable loss function, and \(\varphi:\cX\to\RR\) is a regularization term. An adversary \(y\in\cY:=\{y\in\RR^m:\|y\|_p\le\delta\}\) is allowed to perturb each input \(\zeta_j\), where \(1\le p\le\infty\) and \(\delta>0\) controls the perturbation magnitude. In general, \(\ell(x;\zeta_j+y,\eta_j)\) is nonconvex with respect to \(x\) and nonconcave with respect to \(y\).
			\end{example}
			

			When the function $f$ is convex in $x$ and concave in $y$, classical saddle-point theory provides robust theoretical guarantees and has been extensively studied following von Neumann’s pioneering work~\citep{Neumann1928ZurTD}; see, for example,~\citep{chen2014optimal,Nemirovski2004ProxMethodWR}. Progress has also been achieved under one-sided structural assumptions, particularly the nonconvex--concave or convex--nonconcave settings where $f(x,\cdot)$ is concave for any fixed $x$, or $f(\cdot,y)$ is convex for any fixed~$y$~\citep{xu2024decentralized,xu2023unified,xu2024derivativefreeminMax}. In contrast, the fully nonconvex--nonconcave scenario remains less explored. In this setting, classical minimax results, such as Sion's theorem~\citep{Sion1958OnGM}, do not hold, and saddle points for problem~\eqref{original_minimax_problem} may fail to exist altogether. A common algorithmic approach in this challenging landscape is to focus instead on notions of game-stationarity; see, e.g.,~\citep{jiang2023optimality,li2022nonsmooth,lin2023gradient,xu2023unified}. Although game-stationarity concepts are useful, they may not fully align with the sequential min--max objective, since they do not directly address minimizing the worst-case adversarial response induced by the  decision variable $x$~\citep{chen2024minmax}.

			A natural way to preserve the sequential \emph{min--max} structure is to work with the value-function formulation
			\begin{problemequation}{\ensuremath{\mathrm{P}}}{problemP}
				\label{eq:value_problem}
				\min_{x\in\cX} V(x),\quad \quad
				\text{where } V(x):=\max_{y\in\cY} f(x,y).
			\end{problemequation}
				This formulation captures the sequential decision semantics, but it introduces a   computational bottleneck: evaluating  $V(x)$, or even reliably approximating its first-order information for optimization purposes, typically requires solving a global maximization subproblem over the set $\cY$. When the function $f(x,\cdot)$ is nonconcave, the inner maximization is inherently challenging. 
				Similar computational bottlenecks appear prominently in robust optimization and distributionally robust optimization contexts, where worst-case objective functions, uncertainty sets, or ambiguity sets give rise to intricate inner optimization problems that must nonetheless be resolved efficiently~\citep{Rahimian2022DROFrameworks, Gao2023WassersteinDRSO}.
				
				The standard stochastic approximation tackles the inner maximization by discretizing the set~$\cY$ via sampling. This   idea underpins scenario approximations for uncertain convex optimization and sample-based approximations for probabilistic constraints~\citep{calafiore2005uncertain,LuedtkeAhmed2008SampleApprox}. It plays a vital role in the algorithm design of   chance-constrained, robust, and distributionally robust optimization across diverse applications such as reservoir management, call-center staffing, and Wasserstein-ambiguity modeling~\citep{Gurvich2010CallCenterChance,Rahimian2022DROFrameworks,Gao2023WassersteinDRSO}. Recent studies have developed tighter approximations via compact sample-average formulations, matrix inequalities, and submodularity-based techniques for joint chance constraints~\citep{Porras2023TightCompactSAA,Karimi2021CCLMI,HoNguyen2023StrongLHS,KilincKarzan2022JointChanceSubmod}. Similar approximation techniques also appear in multistage stochastic programming and effective-scenario methods in distributionally robust optimization~\citep{Bodur2022TwoStageLDR,Rahimian2019EffectiveScenarios,xu2018distributionally}.
				
				For problem~\eqref{eq:value_problem}, the standard stochastic approximation replaces $V(x)$ by
				\begin{problemequation}{\ensuremath{\overline{\mathrm{P}}_N}}{problemPbarN}
					\label{eq:PbarN}
					\min_{x\in\cX}
					\left\{
					\overline V_N(x):=\max_{j\in[N]} f(x,y_j)
					\right\},
				\end{problemequation}
					where $Y_N=\{y_1,y_2,\ldots,y_N\}\subset\cY$ and $[N]:=\{1,2,\ldots,N\}$. Let
					\begin{equation}
						\label{eq:beta}
						\beta_N:=\max_{y\in \cY}\min_{j\in[N]}\|y-y_j\|
					\end{equation}
					denote the coverage radius of $Y_N$ in $\cY$. Under smoothness of $f$ in $y$, \citet{jiang2023pure} proved that
					$\sup_{x\in\cX}|V(x)-\overline V_N(x)|=O(\beta_N)$.
					For compact sets in finite-dimensional spaces, however, $\beta_N$ typically decreases only polynomially in $N$, with the rate determined by the ambient or intrinsic dimension of $\cY$ \citep{shapiro2009lectures,billingsley1995probability}. Consequently, the standard stochastic approximation may require a large number of inner samples to achieve high accuracy, a sample-efficiency issue that also motivates tighter sample-average approximations and effective-scenario reductions in adjacent literatures \citep{Porras2023TightCompactSAA,Rahimian2019EffectiveScenarios}.
					
					{These observations motivate a natural question: 
						\begin{center}
							\textit{can each sampled point be made more informative by exploiting local smoothness of the   objective?}
					\end{center}}
					
					In this paper,  rather than sampling $f$ directly, we exploit smoothness in the inner variable to construct a tractable local majorant at each sampled point. Throughout this paper, we make the following assumption. 
					\begin{assumption}\label{ass:smoothness}
						The function $f$ is twice continuously differentiable in   an open neighborhood of $\cX \times \cY$. Moreover, there exist constants $L, L_0>0$ such that, for all $(x,y),(x',y') \in \cX \times \cY$,
						\begin{equation}\label{eq:Lipschtiz_smooth}
							\|\nabla f(x,y)-\nabla f(x',y')\|
							\le L \|(x,y)-(x',y')\|, 	 
						\end{equation}
						\begin{equation}\label{eq:mainLxyy_ass_compact_short}
							\|\nabla^2_{xy} f(x,y)-\nabla^2_{xy} f(x',y')\|
							\le
							L_0\|(x,y)-(x',y')\|.
						\end{equation} 
					\end{assumption}
					
					Under Assumption~\ref{ass:smoothness}, for any $x\in \cX$ and $y,z\in \cY$, we have
					\begin{equation}\label{eq:qz}
						f(x,z)\ge f(x,y)+\langle\nabla_y f(x,y),\,z-y\rangle-\frac{L}{2}\|z-y\|^2.
					\end{equation}
					Now,  we introduce {a  majorant function} $q:\cX\times \cY \to \mathbb{R}$ of $f$ defined by
					\begin{equation}
						\label{eq:q}
						q(x,y)
						:= f(x,y)+\max_{z\in \cY}\Big(\langle\nabla_y f(x,y),\,z-y\rangle-\frac{\widehat{L}}{2}\|z-y\|^2\Big),
					\end{equation}
					where $\widehat{L}>L$ is a constant.
					The inner problem is strongly concave and has the unique solution $\Pi_{\mathcal{Y}}\left(y+\widehat{L}^{-1} \nabla_y f(x, y)\right)$; hence $q$ is continuous by Assumption~\ref{ass:smoothness}.  It is easy to verify that  $$q(x,y) \le V(x), \text{ and }
					f(x,y)\le q(x,y) \text{ for any } (x,y) \in \cX \times \cY,$$ from \eqref{eq:qz} and \eqref{eq:q}.  Hence we have
					$$f(x,y)\le q(x,y) \le V(x)
					\quad \forall(x,y) \in \cX \times \cY.$$
					Moreover, if $y^*\in\arg\max_{y\in\cY}f(x,y)$, then
					$\langle\nabla_y f(x,y^*),\,z-y^*\rangle\le0, \, \forall z\in \cY$, which yields the equalities $q(x,y^*)=f(x,y^*)=V(x)$, and  
					$
					V(x)=\max_{y\in \cY} q(x,y).
					$
					Therefore, \eqref{eq:value_problem} admits the equivalent reformulation
					\begin{equation}\label{eq:Q}
						\min_{x\in\cX} \max_{y\in \cY} q(x,y).
					\end{equation}
					
					Since   $q$ is continuous, discretizing the inner maximization in~\eqref{eq:Q} over the same sample set \mbox{$Y_N=\{y_1,y_2,\dots,y_N\}\subset \cY$} yields the \emph{majorant stochastic approximation}
					\begin{problemequation}{\ensuremath{\mathrm{P}_N}}{problemPN}
						\label{eq:PN}
						\min_{x\in\cX}
						\left\{
						V_N(x):=\max_{j\in[N]} q(x,y_j)
						\right\}.
					\end{problemequation}
						Existing approximation results for continuous max problems (e.g.,~\citep{xu2018distributionally}) yield a nonasymptotic bound between~\eqref{eq:Q} and~\eqref{eq:PN} in terms of $\beta_N$, which translates directly to problem~\eqref{eq:value_problem} based on the relation between $q$ and $f$. 
						Below we prove a sharper bound.
						It holds that
						\begin{equation}
							\label{eq:Vnthere}
							\overline V_N(x)=\max_{j\in[N]} f(x,y_j)\ \le\ \max_{j\in[N]} q(x,y_j)= V_{N}(x) \ \le\ \max_{y\in \cY} q(x,y)=V(x), \quad \forall x \in \cX.
						\end{equation}
						Hence $V_N$ is a tractable lower approximation of $V$ that is uniformly tighter than $\overline V_N$. Under the same sample set and coverage radius, this refinement improves the approximation error from $O(\beta_N)$ to $O(\beta_N^2)$. Details can be found in Section~\ref{sec:2}.
						This improvement is of \textit{quadratic order}, as the coverage radius $\beta_N$ decreases almost surely to zero with increasing sample size $N$, under an appropriate sampling or covering strategy on the compact set $ \mathcal{Y}$ (e.g., Assumption~\ref{ass:sample} given later).

						\subsection{Contributions}
						We propose a {majorant stochastic approximation} approach for the nonconvex--nonconcave minimax  problem~\eqref{eq:value_problem}.
						Specifically, we construct a majorant function \(q(x,y)\) for the  objective $f(x,y)$, which induces the  {majorant stochastic approximation} \mbox{\( V_N(x)=\max_{j\in[N]} q(x,y_j)\)}, providing a more accurate   alternative to standard {stochastic approximation} \mbox{\(\overline V_N(x)=\max_{j\in[N]} f(x,y_j)\)}. 
						For value functions and optimal values, we prove that the {majorant stochastic approximation} significantly improves the {approximation} accuracy from \(O(\beta_N)\) to \(O(\beta_N^2)\), as the coverage radius of the samples \(\beta_N\) tends to 0 almost surely.   
						Under a Polyak--Łojasiewicz condition, we further establish nonasymptotic error bounds and subsequential consistency results for Clarke stationary points. 

						In  addition, we design a {sequential} smoothing majorant stochastic approximation (SMSA) algorithm with convergence result, enabling efficient solution of the {majorant stochastic approximation} problem \eqref{eq:PN}.  
						The
						method solves a sequence of log-sum-exp smoothed subproblems by projected
						gradient iterations, while increasing the sample size and decreasing the
						smoothing and stationarity parameters.
						Numerical experiments demonstrate the computational efficiency and approximation quality  of the  SMSA algorithm.

						\subsection{Notations and definitions}\label{subsec:1.4-defs}
						Let  $\RR_+ := [0,+\infty)$, $\overline\RR_+ = \RR_+\cup \{+\infty\}$, and $\|\cdot\|$ be the Euclidean norm.
						Denote the closed ball centered at $x$ with radius~$r$ by $\cB(x,r)=\{z\in \mathbb{R}^n : \|x-z\|\le r\}$. For a nonempty closed set $\Omega\subseteq\RR^p$, define
						$
						\dist(x_0,\Omega):=\inf_{x\in\Omega}\|x-x_0\|,\text{ for } x_0\in\RR^p,
						$
						and let $\mathrm{co}(\Omega)$ be the convex hull of $\Omega$.
						$\Pi_{\Omega}$  represents the Euclidean projection onto  $\Omega$.
						If $\Omega$ is compact, define its diameter by
						$
						D_{\Omega} := \sup\{\|y-y'\| : y,y' \in \Omega\}.
						$
						For nonempty sets $A,B\subseteq\RR^p$, their (excess) distance and Hausdorff distance are
						\[
						d(A,B):=\sup_{a\in A}\dist(a,B),\text{ and } d_H(A,B):=\max\{d(A,B),\,d(B,A)\}, \text{ respectively.}
						\]Let $h:\RR^p\to\RR$ be locally Lipschitz. 
						The Clarke subdifferential~\citep{clarke1990optimization} of $h$ at $\bar y$ is
						\begin{equation}\label{eq:clarke_subdifferential}
							\partial h(\bar y):=\big\{z\in\RR^p:\ \langle z,d\rangle\le \limsup_{y\to\bar y,\ t\downarrow0}\frac{h(y+td)-h(y)}{t},\ \ \forall d\in\RR^p\big\}.
						\end{equation}
						For a nonempty closed   convex set $\cY\subseteq\RR^p$, the (Euclidean) normal cone
						to $\cY$ at $\overline y\in\cY$ is defined by
						$
						\mathcal N_{\mathcal Y}(\overline y):=\{\xi\in\mathbb{R}^p:\ \langle \xi,y-\overline y\rangle\leq 0, \,
						\forall y\in\mathcal Y\}.
						$
						
						We   end this subsection with three definitions.
						\begin{definition}[$\epsilon$-Clarke stationary point]\label{def:value-stationary}
							For \(\epsilon\ge 0\), we call \(x^*\in\cX\) an \(\epsilon\)-Clarke stationary point of problem~\eqref{eq:value_problem} if
						$
							\dist\bigl(\partial V(x^*)+\cN_{\cX}(x^*),\,0\bigr)\le \epsilon.
						$
						When $\epsilon=0$, $x^*$ is also called a  {Clarke stationary point}.
						\end{definition}
						
						\begin{definition}[Minimax point]
							\label{def:minimax_pt}
							A point $(x^*,y^*)\in \cX\times \cY$ is called a minimax point of problem (\ref{original_minimax_problem}) if it holds that
							$
							f(x^*,y)  \leq  f(x^*,y^*)  \leq  \max_{y' \in \cY} f(x,y')
							$
							for any $x \in \cX$ and $y \in \cY$.
						\end{definition}
						
						\begin{definition}[Smoothing function~\citep{chen2012smoothing}]\label{def:smoothing}
							Let $h:\RR^{p}\to\RR$ be continuous.
							We call \mbox{$\tilde h:\RR^{p}\times\RR_{+}\to\RR$} a smoothing function of $h$, if for any fixed $\mu>0$,  $\tilde h(\cdot; \mu)$ is continuously differentiable and
							for all $\overline y\in\RR^{p}$, it holds
							$
							\lim_{y\to\overline y, \mu\downarrow0}\tilde h(y; \mu)=h(\overline y).
							$
						\end{definition}

						\subsection{Organization}
						The remainder of the paper is organized as follows. In Section~\ref{sec:2}, we investigate accuracy bounds of the majorant stochastic approximation and the relationships among   \eqref{eq:value_problem}, \eqref{eq:PN} and \eqref{eq:PbarN}.  Section~\ref{sec:algorithm} presents the SMSA method and the convergence analysis for solving problem \eqref{eq:value_problem}. Numerical experiments are reported in Section~\ref{sec:numer}. Conclusion is given in the last section.

						\section{Relationships among problems \texorpdfstring{\eqref{eq:value_problem}, \eqref{eq:PN}, and \eqref{eq:PbarN}}{(P), (P\_N), and (PbarN)}, and  accuracy bounds }\label{sec:2}
						In this section, we compare problem~\eqref{eq:value_problem} with its standard and majorant stochastic approximations, \eqref{eq:PbarN} and \eqref{eq:PN}.
						Our analysis proceeds along three dimensions: optimal values, minimizer sets, and Clarke stationary sets. Moreover, we show that the accuracy bound of the majorant stochastic approximation is $O(\beta_N^2)$ as $\beta_N\downarrow 0$ almost surely.
						
						\subsection{Preliminaries}\label{subsec:preliminaries}
						In this subsection, we collect basic probabilistic facts for the sampled approximation of~$\cY$.
						Let $P$ be a probability measure supported on~$\cY$, and let $y_1,y_2,\dots,y_N$ be i.i.d.\ samples drawn from~$P$.
						Denote the resulting sample set by $Y_N:=\{y_1,y_2,\ldots,y_N\}$.
						We impose the following  {lower-mass} condition, which is standard in analyses of sampled max discretizations.
						\begin{assumption}\label{ass:sample}
							There exist positive constants $C_0$, $\gamma$ and $\delta_0$ such that
							$$
							P\big(\{y\in\cY:\|y-y_0\|\le \delta\}\big) \ge C_0\,\delta^\gamma,
							$$
							for all $y_0 \in \cY$ and $\delta \in (0, \delta_0)$.
						\end{assumption}
						\begin{remark}
							Assumption~\ref{ass:sample} is widely used to control the coverage radius induced by i.i.d.\ samples; see,
							e.g.,~\citep{anderson2014confidence,xu2018distributionally}.
							It is satisfied, for instance, if $P$ admits a density (with respect to Lebesgue measure on $\RR^m$) that is bounded away from zero on~$\cY$, together with a mild volume regularity of~$\cY$ ensuring that
							$\mathrm{vol}(B(y_0,\delta)\cap \cY)\sim \delta^m$ uniformly over $y_0\in\cY$ for small~$\delta$.
							In particular, when $P$ is the uniform distribution on a full-dimensional set~$\cY$ with such regularity, one may take $\gamma=m$ and $C_0$ proportional to $1/\mathrm{vol}(\cY)$.
						\end{remark}

						Recall the coverage radius (defined in~\eqref{eq:beta})
						$
						\beta_N:=\max_{y\in\cY}\min_{j\in[N]}\|y-y_j\|.
						$
						The next lemma gives an exponential tail bound for~$\beta_N$.
						\begin{lemma}\label{lem:beta_N_key}
							Suppose   Assumption~\ref{ass:sample} holds.
							Then, for any $\epsilon \in (0,\delta_0)$, there exist constants $C_1(\epsilon) > 0$ and $C_2(\epsilon) > 0$ such that
							\begin{equation}
								\label{eq:beta_tail}
								\text{Prob}(\beta_N > \epsilon) \leq C_1(\epsilon) e^{-C_2(\epsilon) N},
							\end{equation}
							where $\text{Prob}$ denotes the probability measure induced by the i.i.d. samples $(y_1,y_2,\ldots,y_N)$ on $\cY^N$.
						\end{lemma}

					Lemma~\ref{lem:beta_N_key} follows from a standard $\epsilon$-net construction combined with a union bound; see,
					e.g.,~\citep[Section~3]{xu2018distributionally} for a closely related proof template, and~\citep[Section~3]{jiang2023pure} for analogous approximation arguments in sequential min--max settings.
					Moreover, \eqref{eq:beta_tail} implies exponential concentration of $\beta_N$ for any fixed $\epsilon>0$, and hence $\beta_N\to 0$ almost surely by the Borel--Cantelli lemma.
					In the subsequent sections, we express nonasymptotic approximation errors in terms of~$\beta_N$ and then use~\eqref{eq:beta_tail} to convert them into almost-sure convergence statements.

					{\subsection{Relationships among value functions and minimizer sets}\label{subsec:value-relations}
						In this subsection, we compare the standard stochastic approximation $\overline{V}_N$ with the majorant stochastic approximation $ {V}_N$. We first record the pointwise ordering among the three value functions and then derive nonasymptotic bounds for optimal values and minimizer sets.
						
						For later use, we also fix a Lipschitz constant for the function values of \(f\) on \(\cX\times\cY\). Since \(f\) is continuously differentiable, the quantity
						$
						L_f:=\sup_{(x,y)\in\cX\times\cY}\|\nabla f(x,y)\|
						$
						is finite. 
						By replacing the constant \(L\) in Assumption~\ref{ass:smoothness} with \(\max\{L,L_f\}\), and then choosing \(\widehat L>L\) accordingly, we use the same symbol \(L\) for both the gradient-Lipschitz and function-value Lipschitz bounds throughout the sequel. 
						
						Fix a sample set $Y_N=\{y_1,y_2,\ldots,y_N\}\subset \cY$.
						Since $\cY$ is compact and $f$ is continuous, $V$ is continuous by Berge's maximum theorem~\citep{Berge1963}, while  $\overline V_N$ and $V_N$ are finite maxima of continuous functions.
						Define the minimizer sets and optimal values of \eqref{eq:value_problem}, \eqref{eq:PbarN}, and \eqref{eq:PN} by
						\[
						\cX^*:=\arg\min_{x\in\cX} V(x),\qquad
						\overline{\cX}_N^*:=\arg\min_{x\in\cX}\overline V_N(x),\qquad
						\cX_N^*:=\arg\min_{x\in\cX} V_N(x),
						\]
						and
						\[
						\nu^*:=\min_{x\in\cX} V(x),\qquad
						\overline \nu_N^*:=\min_{x\in\cX}\overline V_N(x),\qquad
						\nu_N^*:=\min_{x\in\cX} V_N(x).
						\]
						Since $\cX$ is compact and each value function is continuous, $\cX^*$, $\overline{\cX}_N^*$, and $\cX_N^*$ are all nonempty.
						We know from \eqref{eq:Vnthere} that $V_N$ is sandwiched between $\overline V_N$
						and the value function~$V$.
						Minimizing them over $x\in\cX$ gives
						$
						\overline \nu_N^*\ \le\ \nu_N^*\ \le\ \nu^*.
						$
						We first recall the asymptotic consistency implied by existing   results. Under Assumption~\ref{ass:sample}, Lemma~\ref{lem:beta_N_key} implies that $\beta_N\to 0$ almost surely. Combining this fact with the consistency result of~\citet[Proposition~2]{jiang2023pure}, we obtain
						\begin{equation}
							\label{eq:asymp-values}
							\lim_{N \to \infty} |\nu^{*} - {\nu}_N^{*}| =
							\lim_{N \to \infty}|\nu^{*} - {\overline{\nu}}_N^{*}| = 0, \text{ almost surely},
						\end{equation}
						and the associated minimizer sets obey
						\begin{equation}
							\label{eq:asymp-argmins}
							\lim_{N \to \infty} d(\cX_N^{*},\cX^{*}) =
							\lim_{N \to \infty} d(\overline{\cX}_N^{*},\cX^{*}) = 0,  \text{ almost surely}.
						\end{equation}
						
						We next quantify the  approximation quality, i.e., the rates at which the optimal-value and minimizer-set errors vanish  as a function of the coverage radius~$\beta_N$ given in \eqref{eq:beta}.  
						Define
						\mbox{$
							A(\tau):=\big\{x\in\cX:\ \dist(x,\cX^*)\ge \tau\big\}
							$}
						for any $\tau\ge 0$.
						We   define a growth modulus $\cR:\RR_+\to\overline\RR_+$ by
						\begin{equation}\label{eq:R-def}
							\cR(\tau):=
							\begin{cases}
								\displaystyle \inf\big\{\,V(x)-\nu^*:\ x\in A(\tau)\,\big\}, & \text{if $A(\tau)\neq\emptyset$},\\[0.4em]
								+\infty, & \text{otherwise}.
							\end{cases}
						\end{equation}
						Whenever $A(\tau)$ is nonempty, compactness of $\cX$ and continuity of $V$ ensure that the infimum in
						\eqref{eq:R-def} is attained. It holds that $\cR(\tau)>0$ for every $\tau>0$ because
						$V(x)-\nu^*\ge 0$ on $\cX$ and equality holds only on $\cX^*$.
						Moreover, $\cR(\tau)$ is nondecreasing in~$\tau$ since $A(\tau)$ shrinks as $\tau$ increases.
						We then define the generalized inverse $\cR^{-1}(t)$ by
						\begin{equation}\label{eq:R-inv-def}
							\cR^{-1}(t):=\sup\big\{\tau\ge 0:\ \cR(\tau)\le t\big\}.
						\end{equation}
						By construction, $\cR^{-1}$ is well-defined, nondecreasing, right-continuous, and satisfies
						$\cR^{-1}(t)\to 0$ as $t\downarrow 0$.
						These moduli have been utilized in stability and error-bound analysis; see, e.g.,~\citep[Chap.~7]{roc1998var}.
						
						Existing results~\citep[Proposition~2]{jiang2023pure} for the {standard stochastic approximation} yield first-order bounds:
						$
						|\nu^*-\overline\nu_N^*|\ \le\ L\beta_N,
						\text{ and }
						d(\overline{\cX}_N^*,\cX^*) \ \le\ \cR^{-1}\!\big(L\beta_N\big).
						$
						Because it holds that $\overline V_N(x) \le V_{N}(x) \le  V(x),$ the same coarse order also applies to $V_N$. These bounds, however,  do not exploit the curvature encoded in $q$. The next theorem shows that the majorant approximation is substantially sharper: its error is governed by a second-order term and improves from $O(\beta_N)$ to $O(\beta_N^2)$. The proof is given in Appendix~\ref{appen:1}.
						

						\begin{theorem}\label{thm:value solution-gap}
							Suppose that Assumption~\ref{ass:smoothness} holds. 
							For any   sample set \mbox{$Y_N=\{y_1,y_2,\ldots,y_N\}\subset \cY$}, the majorant approximation satisfies the uniform pointwise bound
							\begin{equation}
								0\leq V(x) - V_N(x) \leq \widehat{L}\beta_N^2,\quad \forall x\in\cX.
							\end{equation}
							Consequently, the optimal values and the associated minimizer sets satisfy
							\begin{equation}\label{eq:value-gap-orders-new}
								0 \leq \nu^*-\nu_N^* \leq \widehat{L}\beta_N^2, \text{ and }d({\cX}_N^*,\cX^*) \leq \cR^{-1}\!\big(\widehat{L}\beta_N^2\big).
							\end{equation}
						\end{theorem}

						%
						
						\begin{remark}
							By  \citet[Theorem 2.2]{chen2024minmax}  and the relation
							$
							V(x)=\max_{y\in \cY} f(x,y)=\max_{y\in \cY} q(x,y),
							$
							every minimizer of the value-function problem induces a minimax point of the original problem.
							More precisely, if $x^*$ is a minimizer of ~\eqref{eq:value_problem}, and \(y^*\in \argmax_{y\in \cY} f(x^*,y)\), then
							\((x^*,y^*)\) is a minimax point of problems~\eqref{original_minimax_problem} and~\eqref{eq:Q}. Thus, Theorem~\ref{thm:value solution-gap} shows that the majorant
							stochastic approximation consistently recovers the \(x\)-component of the minimax set through
							the   {almost-sure} convergence of \(\cX_N^*\) to \(\cX^*\).
							
							For each positive integer \(N\), choose \(x_N^*\in{\cX}_{N}^*\) and
							\(\overline x_N^*\in \overline{\cX}_{N}^*\).  
							Let $y_N^*\in \arg\max_{y\in Y_N}q(x_N^*,y)$ and  $\overline y_N^*\in \arg\max_{y\in Y_N}f(\overline x_N^*,y)$.
							From the bounds
							$$|V(x_N^*)-\max_{y\in Y_N}q(x_N^*,y)|\le \widehat{L}\beta_N^2, \,\,  |V(\overline x_{N}^*)-\max_{y\in Y_N}f(\overline x_{N}^*,y)|\le L\beta_N,
							$$
							and Lemma~\ref{lem:beta_N_key}, there exist indices $N_j\uparrow\infty$, $\overline x_{N_j}^*\in \overline{\cX}_{N_j}^*$, $x_{N_j}^*\in{\cX}_{N_j}^*$, and limit points $(\overline x^*,\overline y^*)$ and $(x^*,y^*)$ such that
							$$
							(\overline x_{N_j}^*,\overline y_{N_j}^*)\to (\overline x^*,\overline y^*),
							(x_{N_j}^*,y_{N_j}^*)\to (x^*,y^*), \text{ as } j\to\infty, \text{ almost surely}.
							$$
							Moreover, with probability one, both limits $(\overline x^*,\overline y^*)$ and $(x^*,y^*)$ are minimax points of~\eqref{original_minimax_problem}.
						\end{remark}

						{\subsection{Relationships among stationary points}\label{sub:stationary_points}
							In this subsection, we investigate the relationships among  Clarke stationary sets associated with $V$,  $\overline V_N$ and $V_N$.
							Since all three value functions are locally Lipschitz on $\cX$, their Clarke subdifferentials  are well defined for every $x\in\cX$.
							For a locally Lipschitz function $\phi:\RR^n\to\RR$, define  the Clarke stationary set over $\cX$ by
							\[
							\cD(\phi)\ :=\ \big\{x\in\cX:\ 0\in \partial \phi(x)+\cN_{\cX}(x)\big\}.
							\]
							We write $\cD:=\cD(V)$, $\overline{\cD}_N:=\cD(\overline V_N)$, and $\cD_N:=\cD(V_N)$.

							Under Assumption~\ref{ass:sample}, we have \(\beta_N\to0\) almost surely, and the standard consistency theory implies a \emph{subdifferential consistency} property: whenever \(x_N\to x\), \(g_N\in\partial \overline V_N(x_N)\) or \(g_N\in\partial V_N(x_N)\), and the sequence \(\{g_N\}\) is bounded, every accumulation point of \(\{g_N\}\) belongs to \(\partial V(x)\); see, e.g., \citet[Section~3]{xu2018distributionally}. 
							Combining this subdifferential consistency with the closed-graph property of the normal cone mapping
							yields the following proposition, whose proof is given in Appendix~\ref{appen:3}.
							
							\begin{proposition}[Asymptotic relationships among Clarke stationary points]\label{thm:stationary_sets_asymp2}
								Suppose Assumptions~\ref{ass:smoothness}--\ref{ass:sample} hold. Let $\{x_N\}\subseteq \cX$ satisfy $x_N\to x$ as $N\to\infty$.
								If either $x_N\in \overline{\cD}_N$ for all $N$  or $x_N\in \cD_N$ for all $N$, then with probability one,
								the limit point $x$ belongs to $\cD$, i.e.,
								$
								0\in \partial V(x)+\cN_{\cX}(x).
								$
							\end{proposition}
							
							We next derive a nonasymptotic bound for the majorant stationary set. The standard approximation admits a parallel but weaker bound, discussed in Remark~\ref{rem:4}.
							The key ingredient is a pointwise Polyak--Łojasiewicz (PL) condition for the inner maximization problem. This condition yields a quadratic-growth estimate for \(y\mapsto f(x,y)\) around the exact maximizer set
							\mbox{$
								S_0(x)=\arg\max_{y\in \mathcal{Y}} f(x,y),
								$}
							and, together with Theorem~\ref{thm:value solution-gap}, leads to an explicit approximation bound.
							
							\begin{assumption}[PL condition]
								\label{ass:PL_f_global}
								For any $x\in\cX$, there exists a constant \mbox{$\kappa_p(x)>0$} such that for all $y\in\cY$,
								\begin{equation}\label{eq:PL_f_global}
									\kappa_p(x)\big(V(x)-f(x,y)\big)
									\ \le\
									\dist^2\!\Big(0,\,-\nabla_y f(x,y)+\cN_{\cY}(y)\Big).
								\end{equation}
							\end{assumption}

							Before proceeding to our main results, we establish the quadratic growth property of $q(x,\cdot)$ for all $x\in\cX$. The proof is given in Appendix~\ref{appen:4}.
							
							\begin{lemma}
								\label{lem:QG_f_to_QG_q_global}
								Suppose that Assumption~\ref{ass:PL_f_global} holds and $\widehat{L}>L$. For any $x\in\cX$, define
								\begin{equation}\label{eq:mu_q_global}
									\kappa_q(x):=\frac{\kappa_p(x)(\widehat{L}-L)}{2\kappa_p(x)+8(\widehat{L}-L)}.
								\end{equation}
								Then, for all $x\in\cX$ and $y\in\cY$, it holds that
								\begin{equation}\label{eq:QG_q_global}
									V(x)-q(x,y)\ \ge\ \kappa_q(x)\,\dist^2\big(y,S_0(x)\big),
								\end{equation}
								where $S_0(x) := \arg\max_{y\in \cY} f(x,y)$.
							\end{lemma}

							\begin{remark}[Necessity of the strict inequality $\widehat{L}>L$]
								The role of \mbox{$\widehat{L}>L$} is to create a strictly positive curvature gap in the construction of $q$. If one simply sets $\widehat{L}=L$, then this additional curvature disappears, and one can no longer guarantee a strictly positive quadratic-growth modulus for $q(x,\cdot)$ in general. We   illustrate this point with the following example. Let
								\[
								\cY=[-1,1], \qquad f(x,y)=-y^2+\frac14 y^4,
								\]
								where $f$ is independent of $x$ and $f(x,\cdot)$ is nonconcave on $\cY$. Since
								$
								\max_{y\in\cY}\left|\frac{\partial^2 f}{\partial y^2}(x,y)\right|=2,
								$
								we may take $L=2$. For every $x\in\cX$,  
								$
								V(x)=\max_{y\in\cY} f(x,y)=0,
								\text{ and }
								S_0(x)=\arg\max_{y\in\cY} f(x,y)=\{0\}.
								$
								We next verify that the PL condition in Assumption~\ref{ass:PL_f_global} holds.
								Because
								$
								\nabla_y f(x,y)=-2y+y^3,
								$
								we have, for $|y|<1$,
								$
								\dist^2\!\Bigl(0,-\nabla_y f(x,y)+\cN_{\cY}(y)\Bigr)
								=(2y-y^3)^2,
								$
								while at the boundary points,
								\[
								\dist^2\!\Bigl(0,-\nabla_y f(x,1)+\cN_{\cY}(1)\Bigr)=1,
								\qquad
								\dist^2\!\Bigl(0,-\nabla_y f(x,-1)+\cN_{\cY}(-1)\Bigr)=1.
								\]
								On the other hand,
								$
								V(x)-f(x,y)=y^2-\frac14 y^4.
								$
								For $|y|=1$, $
								\frac43(V(x)-f(x,y))=1.
								$
								For $|y|<1$,
								\[
								(2y-y^3)^2-\frac43\Bigl(y^2-\frac14 y^4\Bigr)
								=
								\frac13 y^2(1-y^2)(8-3y^2)\ge 0.
								\]
								Hence Assumption~\ref{ass:PL_f_global} holds with
								$
								\kappa_p(x)\equiv \frac43.
								$
								However, if we set $\widehat{L}=L=2$, then for any $y\in\cY$,
								the maximizer in the definition of $q$ is attained at
								$z=y+\frac1L\nabla_y f(x,y)=\frac{y^3}{2}$, and therefore
								\[
								q(x,y)
								=
								f(x,y)+\frac{1}{2L}\|\nabla_y f(x,y)\|^2
								=
								-y^2+\frac14 y^4+\frac{(-2y+y^3)^2}{4}
								=
								-\frac34 y^4+\frac14 y^6.
								\]
								Consequently, it holds that 
								$
								V(x)-q(x,y)=\frac14 y^4(3-y^2),
								\text{ and }
								\dist^2\bigl(y,S_0(x)\bigr)=y^2.
								$
								It follows that
								\[
								\frac{V(x)-q(x,y)}{\dist^2(y,S_0(x))}
								=
								\frac14 y^2(3-y^2)\to 0, \text{ as } y\to 0.
								\]
								Thus, there is  no constant $\kappa_q(x)>0$ such that the inequality
								$
								V(x)-q(x,y)\ge \kappa_q(x)\,\dist^2\bigl(y,S_0(x)\bigr)
								$
								holds for all $y\in\cY$.
								Furthermore, even in any small neighborhood of the solution set $\{0\}$, the above fact does not hold.
								This shows that the strict inequality $\widehat{L}>L$ is generally necessary if one wishes to obtain a nontrivial positive quadratic-growth bound for $q(x,\cdot)$.
							\end{remark}

							With the preceding lemma established, we can now bound the distance between the subdifferential sets  $\partial V(x)$ and   $\partial V_N(x)$.  The proof is provided in Appendix~\ref{appen:4}.
							\begin{theorem}
								\label{thm:main_betaN_global}
								Suppose that Assumptions~\ref{ass:smoothness}--\ref{ass:PL_f_global} hold, and $\widehat{L}>L$.    Then, for any  $N \ge 1$,  it holds that
								\begin{equation}\label{eq:main_bound_beta_global}
									d\big(\partial V_N(x),\partial V(x)\big)
									\ \le\
									L_{\kappa}(x) \beta_N, \quad \forall x\in\cX, 
								\end{equation}
								where the $x$-dependent factor $L_{\kappa}(x)$ is defined as
								$
								L_{\kappa}(x) := \left(D_{\cY}L_0 + \left(3+\frac{L}{\widehat{L}}\right)L\right)\sqrt{\frac{\widehat{L}}{\kappa_q(x)}}.
								$
							\end{theorem}

							For any $x\in \cX$, define the Clarke-stationarity residuals of \eqref{eq:value_problem} and \eqref{eq:PN} by
							\begin{equation}
								\label{eq:rrn}
								\begin{aligned}
									r(x) &:=\mathrm{dist}\big(0,\partial V(x)+  {\cN}_{\cX}(x)\big),\quad
									r_N(x) :=\mathrm{dist}\big(0,\partial V_N(x)+  {\cN}_{\cX}(x)\big).
								\end{aligned}
							\end{equation}
							To translate bounds on the subdifferential sets into bounds on the distance to the stationary set, we follow~\cite{liu2014quantitative} by defining the level set $A_{\mathcal{D}}(\tau):= \{x\in\cX : \, \dist(x,\mathcal{D})\ge \tau\}$ for any $\tau \ge 0$. We then introduce the error bound modulus $\mathcal{R}_{\mathcal{D}}: \mathbb{R}_+ \to \overline{\mathbb{R}}_+$ as
							\begin{equation}\label{eq:Rs-def}
								\cR_{\cD}(\tau):=
								\begin{cases}
									\displaystyle \inf \{\,r(x):\, x\in A_{\cD}(\tau) \,\}, & \text{if $A_{\cD}(\tau)\neq\emptyset$},\\
									+\infty, & \text{otherwise}.
								\end{cases}
							\end{equation} 
							Since \(V\) is continuous on the compact set \(\cX\), it has a global minimizer, say \(\bar x\in\cX\). Moreover, \(V\) is locally Lipschitz and \(\cX\) is closed and convex. Clarke's necessary optimality condition at \(\bar x\) therefore yields
							$
							0\in\partial V(\bar x)+\cN_{\cX}(\bar x).
							$
							Hence \(\cD\neq\emptyset\). Moreover, the Clarke subdifferential and normal-cone mappings have closed graphs, which implies the lower semicontinuity of \(r\). Therefore, the infimum \(\cR_{\cD}(\tau)>0\) over the compact set \(A_{\cD}(\tau)\) is positive for every \(\tau>0\).
							Its generalized inverse $\mathcal{R}_{\mathcal{D}}^{-1}:\mathbb{R}_+\to \mathbb{R}_+$ is given by
							\begin{equation}\label{eq:Rs-inv-def}
								\cR_{\cD}^{-1}(t)
								:= \sup\{\, \tau \ge 0 : \cR_{\cD}(\tau) \le t \,\}.
							\end{equation}
							By construction, $\mathcal{R}_{\mathcal{D}}^{-1}$ is well-defined, nondecreasing, right-continuous, and satisfies $\mathcal{R}_{\mathcal{D}}^{-1}(t)\downarrow 0$ as $t\downarrow 0$. The following proposition gives a bound on the distance to the stationary set, whose proof is given in Appendix~\ref{appen:5}.
							
							{\begin{proposition} \label{thm:conv_stat}
									Under the conditions of Theorem~\ref{thm:main_betaN_global}, for any stationary point $x_N \in \mathcal{D}_N$, its distance to $\mathcal{D}$ can be bounded by
									$
									\dist(x_N, \mathcal{D}) \leq \mathcal{R}_{\mathcal{D}}^{-1}\!\big(L_{\kappa}(x_N)\beta_N \big).
									$
									Consequently, it holds that
									$
									d({\cD}_N,\cD) \leq \cR_{\cD}^{-1}\big(\sup_{x \in \mathcal{D}_N} L_{\kappa}(x)\beta_N \big).
									$
								\end{proposition}
								
								\begin{remark}\label{rem:4}
									The term
									\(\sup_{x\in\cD_N} L_{\kappa}(x)\) in Proposition~\ref{thm:conv_stat} is automatically finite whenever the PL modulus in
									Assumption~\ref{ass:PL_f_global} is uniformly bounded away from zero on~\(\cX\).
									Indeed, if it holds that $
									\underline{\kappa}_p:=\inf_{x\in\cX}\kappa_p(x)>0,$
									then \eqref{eq:mu_q_global} implies that
									$
									\kappa_q(x)\ge
									\underline{\kappa}_q:=
									\frac{\underline{\kappa}_p(\widehat{L}-L)}
									{2\underline{\kappa}_p+8(\widehat{L}-L)}
									>0,\ \forall x\in\cX.
									$
									Consequently, it holds
									\[
									L_{\kappa}(x)\le
									L_{\kappa}^{\max}:=
									\left(D_{\cY}L_0+\left(3+\frac{L}{\widehat{L}}\right)L\right)
									\sqrt{\frac{\widehat{L}}{\underline{\kappa}_q}},
									\qquad \forall x\in\cX.
									\]
									Since
									\(\cD_N\subset\cX\) is nonempty, \(\sup_{x\in\cD_N}L_{\kappa}(x)\) is well defined and satisfies
									\[
									\sup_{x\in\cD_N}L_{\kappa}(x)\le L_{\kappa}^{\max}<\infty,
									\qquad \forall N\ge 1.
									\]
									Then, Proposition~\ref{thm:conv_stat} simplifies to the uniform estimate
									$
									d(\cD_N,\cD)\le \cR_{\cD}^{-1}\!\big(L_{\kappa}^{\max}\beta_N\big).
									$
									
									Under the same assumptions, an entirely parallel argument for the standard stochastic
									approximation \(\overline V_N\) yields
									\[
									d(\overline{\cD}_N,\cD)
									\le
									\cR_{\cD}^{-1}\!\Big(
									\sup_{x\in\overline{\cD}_N}L_{\rm dir}(x)\sqrt{\beta_N}
									\Big),
									\]
									where
									$
									L_{\rm dir}(x):=L\sqrt{\frac{8L}{\kappa_p(x)}}.
									$
									Therefore, the majorant approximation improves the dependence on the coverage radius from  \(\sqrt{\beta_N}\)  to order \(\beta_N\) inside the error-bound modulus.
							\end{remark}}

							\section{A sequential smoothing majorant stochastic approximation method}\label{sec:algorithm} 
							
							In this section, we propose a sequential smoothing majorant stochastic
							approximation (SMSA) method that approximately solves a sequence of smoothed
							majorant subproblems. For a prescribed terminal sample cap \(\overline{N}\), the
							method first draws a   sample set
							$
							Y_{\overline{N}}:=\{y_1,\ldots,y_{\overline{N}}\}
							$
							of i.i.d. samples from the sampling distribution on \(\cY\). At each outer
							iteration, the method uses a larger prescribed prefix of this sample set,
							decreases the smoothing parameter, and approximately minimizes the corresponding
							smoothed majorant subproblem over~\(\cX\) by projected gradient. We first present
							the algorithmic framework and then establish its convergence properties.

							
							
							
							\subsection{Algorithm framework}\label{subsec:3.2-framework}
							The ultimate objective is the value function \(V(x)=\max_{y\in \cY} f(x,y)=\max_{y\in \cY} q(x,y)\), whereas the algorithm operates on the majorant stochastic approximation
							$
							V_N(x)=\max_{i\in[N]} q(x,y_i) 
							$
							at each iteration. 
							For each fixed sample $y_i\in Y_N$, Danskin's theorem together with Assumption~\ref{ass:smoothness} implies that the mapping $x\mapsto q(x,y_i)$ is differentiable on $\cX$.   Nevertheless, the function \( V_N \) itself is generally nonsmooth, limiting the direct applicability of first-order gradient methods. To overcome this difficulty, we smooth \(V_N\) by the standard log-sum-exp (LSE) approximation \citep{burke2020subdifferential,shapiro2009lectures}. For any \(\mu>0\), define
							\begin{equation}\label{eq:lse-def-3.1}
								\widetilde V_N(x,\mu)
								:=\mu\log \Big(\sum_{i=1}^N \exp \big(q(x,y_i)/\mu\big)\Big),
							\end{equation}
							where the parameter \(\mu\) controls the bias introduced by smoothing.
							
							To control both the finite-sample approximation error and the smoothing bias,
							we adopt a finite-budget continuation scheme. Specifically, we first prescribe
							\(\overline{N}\ge 2\).  We then generate an increasing sequence
							\(\{N_k\}_{k=0}^{K}\) satisfying
							$
							2\le N_0 < N_1 < \cdots < N_{K}=\overline{N}.
							$ Equivalently, one may use any
							update rule satisfying
							$
							N_{k+1}\in\{N_k+1,\ldots,\overline{N}\},  \text{ whenever } N_k<\overline{N}.
							$
							We then draw a sample set
							\begin{equation}
								\label{eq:YbarN}
								Y_{\overline{N}}:=\{y_1,\ldots,y_{\overline{N}}\}, \text{ with }
								y_i\stackrel{\mathrm{i.i.d.}}{\sim}\mathbb{P}, \text{ for all } i=1,\ldots,\overline{N},
							\end{equation}
							where \(\mathbb{P}\) is the sampling distribution on \(\cY\). For each stage \(k\), the
							active sample set is the prescribed prefix
							\begin{equation}\label{eq:ynk}
								Y_{N_k}:=\{y_1,\ldots,y_{N_k}\}\subseteq Y_{\overline{N}},
								\qquad k=0,\ldots,K .
							\end{equation}
							Consequently,
							$
							Y_{N_0}\subseteq Y_{N_1}\subseteq\cdots\subseteq Y_{N_K}=Y_{\overline{N}}.
							$
							In particular, for each fixed \(N_k\), the elements
							of \(Y_{N_k}\) are i.i.d. \(\mathbb{P}\)-distributed samples. 
							Let \(\{\mu_k\}_{k=0}^K\) and \(\{\epsilon_k\}_{k=0}^K\) be positive nonincreasing sequences satisfying
							$
							0<\mu_k\le \epsilon_k\le 1 .
							$
							At outer iteration \(k\), the method approximately solves the smoothed subproblem
							\[
							\min_{z\in\cX}\ \widetilde V_{N_k}(z,\mu_k)
							\]
							by projected gradient, using \(x^{(k)}\) as a warm start. The outer loop stops
							once \(N_k=\overline{N}\).

							Lemma~\ref{lem:lse-basic}(a) below shows that the smoothed function \(\widetilde{V}_{N_k}(\cdot,\mu_k)\) uniformly upper-approximates \(V_{N_k}\) with an approximation error bounded by \(\mu_k\log N_k\). Meanwhile, Lemma~\ref{lem:lse-basic}(c) demonstrates that the corresponding gradient Lipschitz constant scales inversely with \(\mu_k\). Hence \(\mu_k\) should decrease gradually: taking \(\mu_k\) too large yields poor approximation accuracy, whereas taking it too small makes the subproblem increasingly ill-conditioned. We therefore solve each smoothed subproblem only to a prescribed   first-order tolerance \(\epsilon_k\), measured by
							\begin{equation}\label{eq:residual-def-3.2}
								r_k(z):=\dist\bigl(0,\nabla_z\widetilde{V}_{N_k}(z,\mu_k)+\mathcal{N}_{\mathcal{X}}(z)\bigr).
							\end{equation}
							Once the residual satisfies \(r_k(z)\le \epsilon_k\), the inner loop terminates, and the resulting iterate is carried forward to initialize the next outer iteration. Algorithm~\ref{alg:growing-N} provides a formal description of the procedure, and finite termination of the inner loop is established later in Lemma~\ref{thm:inner_complexity}(b).
							
							\begin{algorithm}[htbp!]
								\caption{A sequential smoothing majorant stochastic approximation method}
								\label{alg:growing-N}
								\begin{algorithmic}[1]
									
									\State \textbf{Input:} Initial point \(x^{(0)}\in\cX\), \(\widehat L>L\), \(\overline{N}\ge2\),  $N_0 \ge 2$,
									a sample set $Y_{\overline{N}}$ satisfying~\eqref{eq:YbarN},  positive nonincreasing smoothing parameters \(\{\mu_k\}\), and positive nonincreasing tolerances \(\{\epsilon_k\}\) satisfying \(0<\mu_k\le\epsilon_k\le1\). Let   \(k=0\).
									
									\While{\(N_k\le \overline{N}\)}
									\State 
									Set
									$
									Y_{N_k}
									$
									by~\eqref{eq:ynk}.
									Construct \(\widetilde V_{N_k}(\cdot,\mu_k)\) via~\eqref{eq:lse-def-3.1}.
									\State Initialize the inner loop: set \(z^{(0)}=x^{(k)}\) and
									\(\alpha_k=1/L_{\mu_k}\), where \(L_{\mu_k}\) is defined
									in~\eqref{eq:lse-Lmu-3.1}.
									\For{\(t=0,1,2,\ldots\)}
									\State Compute \(z^{(t+1)}\) by the projected gradient step
									\begin{equation}\label{eq:update-3.15}
										z^{(t+1)}
										=
										\Pi_{\cX}\Big(
										z^{(t)}-\alpha_k\nabla_z\widetilde V_{N_k}(z^{(t)},\mu_k)
										\Big).
									\end{equation}
									\If{\(r_k(z^{(t+1)})\le \epsilon_k\)}
									\State Set \(x^{(k+1)}=z^{(t+1)}\), and terminate the inner loop.
									\EndIf
									\EndFor
									\If{\(N_k=\overline{N}\)}
									\State \textbf{return} \(x^{(k+1)}\) and \(Y_{\overline{N}}\).
									\EndIf
									\State Choose the next sample size
									\(N_{k+1}\in\{N_k+1,\ldots,\overline{N}\}\), and set \(k\leftarrow k+1\).
									\EndWhile
								\end{algorithmic}
							\end{algorithm}
							
						}

						{
							\subsection{Convergence analysis}\label{subsec:3.3-conv}
							In this subsection, we establish the convergence properties of Algorithm~\ref{alg:growing-N}. Our analysis begins with two preliminary lemmas. The first lemma provides a uniform Lipschitz continuity bound for each of the component functions \(\{q(\cdot,y_i)\}_{i=1}^N\); its proof can be found in Appendix~\ref{appen:B1}.  
							\begin{lemma}\label{lem:vi-lipschitz}
								Suppose that Assumption~\ref{ass:smoothness} holds.
								For each $i\in[N]$,  $q(\cdot,y_i)$ is $G_v$-Lipschitz continuous on $\cX$ with
								$
								G_v:=L(1+D_{\cY}).
								$
								Moreover,
								$
								\|\nabla_x q(x,y_i)\|\le G_v$ for all $x\in\cX$ and $i\in[N].
								$
							\end{lemma}
							
							The second lemma summarizes key properties of the log-sum-exp smoothing function \(\widetilde V_N\), which will be used in the subsequent convergence analysis.
							\begin{lemma}\label{lem:lse-basic}
								Suppose that Assumption~\ref{ass:smoothness} holds, and $\widehat{L}>L$. Then, the following statements hold.
								\begin{itemize}
									\item[(a)]
									For any $x\in\cX$ and $\mu>0$,
									\begin{equation}\label{eq:lse-sandwich-3.1}
										V_N(x)\ \le\ \widetilde V_N(x,\mu)\ \le\ V_N(x)+\mu\log N,
										\text{ and }
										\lim_{\mu\downarrow 0}\widetilde V_N(x,\mu)=V_N(x).
									\end{equation}
									
									\item[(b)]
									$\widetilde V_N(\cdot,\mu)$ is differentiable and its gradient is given by
									\[
									\nabla_x \widetilde V_N(x,\mu)
									=\sum_{i=1}^N p_i(x,\mu)\,\nabla_x q(x,y_i),
									\text{ where }
									p_i(x,\mu)
									:=\frac{\exp \big(q(x,y_i)/\mu\big)}{\sum_{j=1}^N \exp \big(q(x,y_j)/\mu\big)}.
									\]
									Moreover, for any $x\in\cX$, and $\mu_1,\mu_2>0$,
									\[
									\big|\widetilde V_N(x,\mu_1)-\widetilde V_N(x,\mu_2)\big|
									\le 2|\mu_1-\mu_2|\log N,
									\text{ and }
									\|\nabla_x \widetilde V_N(x,\mu)\|\le G_v.
									\]
									
									\item[(c)]
									For any $\mu>0$ and any $x,x'\in\cX$,
									\begin{equation}\label{eq:lse-Lmu-3.1}
										\|\nabla_x \widetilde V_N(x,\mu)-\nabla_x \widetilde V_N(x',\mu)\|
										\le L_\mu \|x-x'\|,
										\text{ where }
										L_\mu := D_{\cY}L_0 + \left(1+\frac{L}{\widehat{L}}\right)L  + \frac{2G_v^2}{\mu}.
									\end{equation}
								\end{itemize}
							\end{lemma}
							Statements~(a) and~(b) in Lemma~\ref{lem:lse-basic} represent standard properties of log-sum-exp smoothing.   Statement~(c) is derived from the uniform \((D_{\mathcal{Y}}L_0+(1+L/\widehat{L})L)\)-smoothness of each component function \( q(\cdot,y_j) \); a detailed proof can be found in Lemma~\ref{lem:explicit_Lxqy_short} of Appendix~\ref{appen:4}. For a comprehensive treatment and additional technical insights, readers may refer to~\citep[Lemmas~2.4--2.5]{liu2025stochasticsmoothingframeworknonconvexnonconcave}. Collectively, these results explicitly characterize the bias–smoothness trade-off inherent in LSE smoothing: the approximation error scales as \(\mu\log N\), whereas the gradient Lipschitz constant increases proportionally to \(1/\mu\). This observation provides a theoretical justification for employing the continuation strategy implemented in Algorithm~\ref{alg:growing-N}.
							
							We next analyze the projected-gradient inner loop at a fixed outer iteration. The proof of the following lemma is given in Appendix~\ref{appen:B2}.
							\begin{lemma}\label{thm:inner_complexity}
								Suppose that Assumption~\ref{ass:smoothness} holds, and $\widehat{L}>L$. 
								Fix an outer iteration \(0\le k\le K\), and let \(\{z^{(t)}\}_{t\ge 0}\) be the inner-loop sequence generated by Algorithm~\ref{alg:growing-N}. Then the following statements hold.
								\begin{itemize}
									\item[(a)] For every \(t\ge 0\),
									\begin{equation}\label{eq:inner_descent}
										\widetilde V_{N_k}(z^{(t+1)},\mu_k)-\widetilde V_{N_k}(z^{(t)},\mu_k)
										\le -\frac{L_{\mu_k}}{2}\|z^{(t+1)}-z^{(t)}\|^2.
									\end{equation}
									
									\item[(b)] The inner loop terminates after at most
									\begin{equation}\label{eq:Tk_bound}
										T_k=
										\left\lceil
										\frac{8C}{\mu_k\epsilon_k^2}
										\Big(
										\widetilde V_{N_k}(x^{(k)},\mu_k)-\min_{z\in\cX}\widetilde V_{N_k}(z,\mu_k)
										\Big)
										\right\rceil
									\end{equation}
									iterations, where
									$
									C:=D_{\cY}L_0 + \left(1+\frac{L}{\widehat{L}}\right)L  +2G_v^2.
									$
								\end{itemize}
							\end{lemma}
						}
						
						To establish a connection between approximate stationarity conditions of  \(\widetilde V_N\) and the stationarity conditions of the   nonsmooth function \(V_N\), we define the \(\epsilon\)-active index set by 
						\[
						I_{N,\epsilon}(x)
						:=
						\big\{ i\in[N]\mid q(x,y_i)\ge V_N(x)-\epsilon\big\},
						\]
						along with its corresponding \(\epsilon\)-approximate subdifferential
						\begin{equation}\label{eq:eps-subdiff}
							\partial_\epsilon V_N(x)
							:=
							\operatorname{co}
							\big\{
							\nabla_x q(x,y_i)\mid i\in I_{N,\epsilon}(x)
							\big\}.
						\end{equation}
						Since the index set is finite, it follows that for each fixed \(x\), there exists a threshold \(\bar\epsilon(x)>0\) such that \(I_{N,\epsilon}(x)=I_N(x)\) whenever \(0<\epsilon<\bar\epsilon(x)\), where \(I_N(x):=\{i\in[N]:q(x,y_i)=V_N(x)\}\). 
						Consequently, the approximate subdifferential \(\partial_\epsilon V_N(x)\) reduces to the Clarke subdifferential \(\partial V_N(x)\) as \(\epsilon\downarrow 0\). 
						
						\begin{definition} 
							\label{def:approx-stationary-majorant}
							 Fix \(\epsilon>0\). A point
							\(x\in\cX\) is called an $\epsilon$-approximate Clarke stationary point of
							the majorant stochastic approximation problem~\eqref{eq:PN} if
							$
							\operatorname{dist}\!\left(
							0,\,
							\partial_{\epsilon}V_N(x)+\cN_{\cX}(x)
							\right)
							\le \epsilon .
							$
						\end{definition}
						
						The following lemma quantitatively characterizes how well the gradient \(\nabla_x\widetilde V_N(x,\mu)\) approximates this \(\epsilon\)-subdifferential. Its proof is provided in Appendix~\ref{appen:B3}. 
						\begin{lemma}\label{lem:lse-grad-to-surrogate}
							Suppose that Assumption~\ref{ass:smoothness} holds, and $\widehat{L}>L$. 
							For any fixed sample set \(Y_N\),  \(x\in\cX\),  \(\mu>0\), and  \(\epsilon>0\), it holds that
							\begin{equation}\label{eq:lse-grad-to-surrogate}
								\operatorname{dist}\big(\nabla_x\widetilde V_N(x,\mu), \partial_\epsilon V_N(x)\big) \leq 2 G_v N\exp \Big(-\frac{\epsilon}{\mu}\Big),
							\end{equation}
							where $G_v$ is given in Lemma~\ref{lem:vi-lipschitz}.
						\end{lemma}
						
						We then present a complexity result  in Theorem~\ref{thm:growing_N}, whose  
						proof is given in Appendix~\ref{appen:B4}.
						
						\begin{theorem} 
							\label{thm:growing_N}
							Suppose that Assumption~\ref{ass:smoothness} holds, and \(\widehat L>L\).
							Let \(\epsilon>0\) be prescribed,  $\overline N\ge \frac{4G_v}{\epsilon}$, and
							\(\bar x:=x^{(K+1)}\) be the output returned by
							Algorithm~\ref{alg:growing-N}.
							If it holds that
							$ 
							\epsilon^{(K)}\le \frac{\epsilon}{2}$,    and  
							\mbox{$\mu^{(K)}\le
							\frac{\epsilon^{(K)}}{2\log(\overline N)},
							$}
							then \(\bar x\) is an $\epsilon$-approximate Clarke stationary point of
							 problem~\eqref{eq:PN}, i.e., 
							\[
							\operatorname{dist}\!\left(
							0,\,
							\partial_{\epsilon}V_{\overline N}(\bar x)
							+\cN_{\cX}(\bar x)
							\right)
							\le \epsilon .
							\] 
						\end{theorem}
					  
\begin{remark} 
	\label{rem:limit_eps_stationarity}
	Suppose, in addition, that Assumption~\ref{ass:sample} holds.
	Let 
	\(\{\eta_\ell\}_{\ell\ge1}\) be a positive sequence with
	\(\eta_\ell\to0\), 
	and choose a  sequence $\left\{N_{\ell}\right\}_{\ell \geq 1}$ satisfying
	\mbox{$
	\bar{N}_{\ell} \geq \max \left\{\bar{N}_{\ell-1}+1,\left\lceil\frac{4 G_v}{\eta_{\ell}}\right\rceil\right\},
	$}
	where $G_v$ is given in Lemma \ref{lem:vi-lipschitz}.
	For each \(\ell\), suppose that
	\(\bar x^\ell\in\cX\) satisfies
	$
	\operatorname{dist}\!\left(
	0,\,
	\partial_{\eta_\ell}
	V_{\overline N_\ell}(\bar x^\ell)
	+\cN_{\cX}(\bar x^\ell)
	\right)
	\le \eta_\ell .
	$  
	Then, with probability one, every accumulation point \(x^*\) of
	\(\{\bar x^\ell\}_{\ell\ge1}\) is a Clarke stationary point of
	problem~\eqref{eq:value_problem}; that is,
	$
	0\in\partial V(x^*)+\cN_{\cX}(x^*) .
	$
	The proof is provided in Appendix~\ref{appen:B4}.
\end{remark}

						\section{Numerical experiments}\label{sec:numer} 	
						In this section, we evaluate the majorant   stochastic approximation and the  SMSA method
						on two classes of test problems: a toy minimax example and a real-world robust logistic regression problem. The first experiment isolates the effects of sample sizes on  the resulting approximation errors. The second experiment assesses out-of-sample robustness under structured feature perturbations. 
						 {Throughout this section, \(N\) denotes the sample size of a  
						fixed sampled model \(V_N\) or \(\overline V_N\). When Algorithm~1 is
						used, \(N_k\) denotes the active sample size at stage \(k\), whereas
						\(\overline N\) denotes the prescribed sample size.}
						All experiments are implemented in Python on a MacBook Pro with an Apple M2 Pro chip and 16 GB of memory. Additional implementation details are provided in
						Appendix~\ref{app:numerics}.

						\subsection{A toy example}\label{subsec:toy_example}
						For $m\ge 2$, set $\cY=[-\pi,\pi]^m$ and $\cX=[-1,1]^2$. Let
						\[
						S_o=\Diag(1,0,1,0,\ldots),\quad S_e=\Diag(0,1,0,1,\ldots),\quad
						A_m=\frac{1}{5}\begin{bmatrix}\mathbf{1}^\top S_o\\ \mathbf{1}^\top S_e\end{bmatrix},\quad
						\psi(x)=\begin{bmatrix}x_1^2-x_2^2\\2x_1x_2\end{bmatrix}.
						\]
						With componentwise sine and cosine, consider the following parametric minimax problem
						\begin{equation}\label{eq:toy_family}
							\min_{x\in\cX}\left\{V(x):=\max_{y\in\cY}F_m(x,y)\right\},\text{ where }
							F_m(x,y):=\psi(x)^\top A_my-\mathbf{1}^\top\big(S_o\sin(y)+S_e\cos(y)\big).
						\end{equation}
						This construction provides a toy nonconvex--nonconcave test instance for which the reference value function $V$ and the optimal value
						$
						v^*=\min_{x\in\cX}V(x)
						$
						can be obtained; see Appendix~\ref{app:toy-details}. 
						
						For each fixed \(m\) and each     fixed-model sample size
						\(N\), we first draw  
						$Y_N:=\{y_1,\ldots,y_{N}\}$  with \mbox{\(y_i\sim \mathrm{Unif}([-\pi,\pi]^m)\)} for all $i=1,2,\ldots,N$.
						Using the same  \(Y_N\),  we compare the standard  approximation $\overline V_N(x):=\max_{i\in[N]}F_m(x,y_i)$
						with the   majorant   approximation
						$V_N(x):=\max_{i\in[N]}q_m(x,y_i)$, where $q_m$ is defined by~\eqref{eq:q} with $f=F_m$ and  		$
						\widehat L=1.5.
						$

						\paragraph{Effect of sample size.}
						We first fix $m=5$ and vary $  
						 {N}\in\{20,50,100,200\}.
						$ 
						Figure~\ref{fig:toy-surfaces-fixedm} displays the reference surface $V$ together with the two stochastic approximations $\overline V_{ {N}}$ and $V_{ {N}}$. As $ {N}$ increases, both approximations move closer to the reference surface. The majorant approximation $V_{ {N}}$, however, tracks $V$ more closely under the same sample set.
						
						\begin{figure}[htbp]
							\centering
							\begin{tabular}{@{}cccc@{}}
								\includegraphics[width=0.235\linewidth]{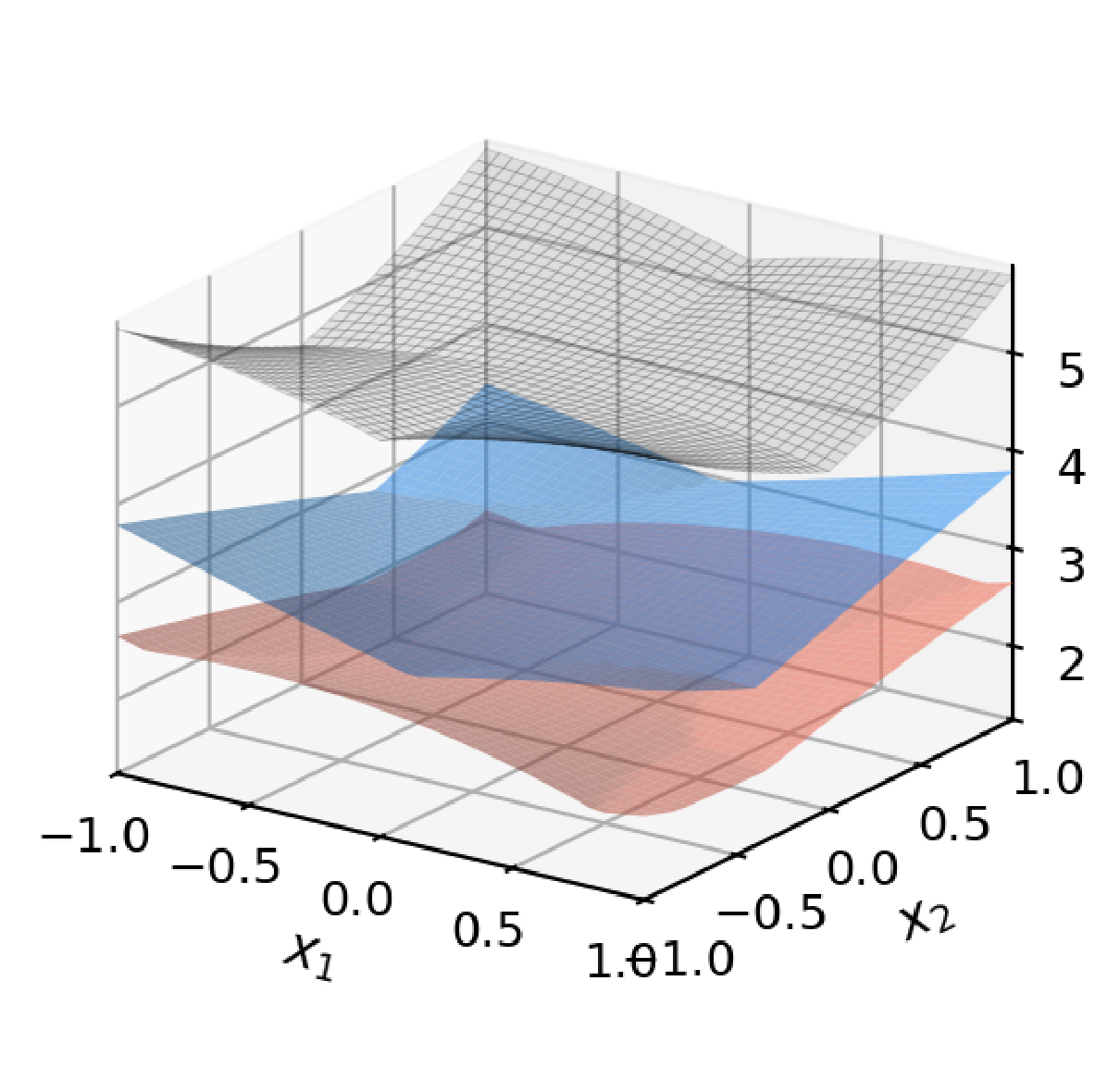} &
								\includegraphics[width=0.235\linewidth]{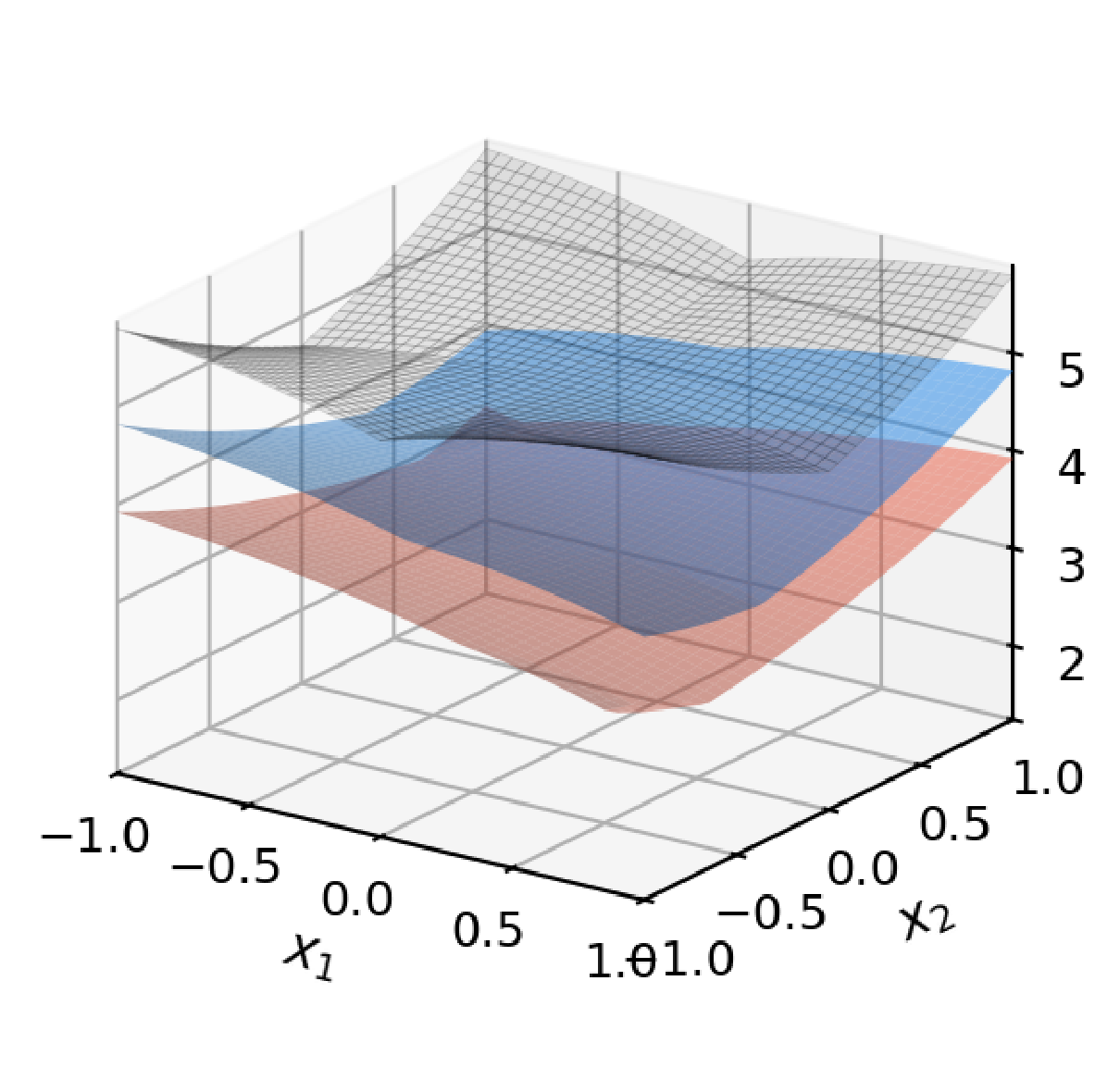} &
								\includegraphics[width=0.235\linewidth]{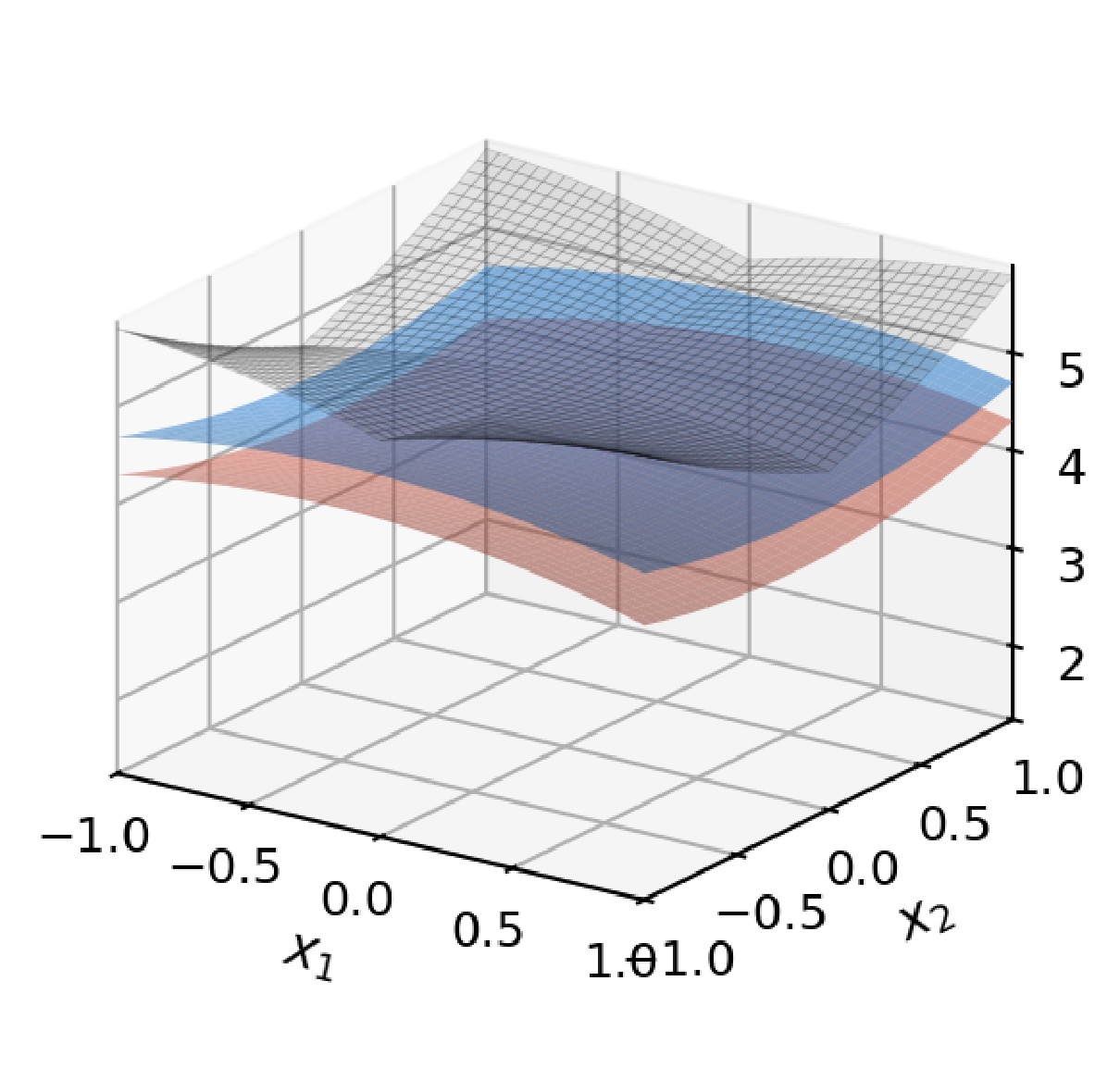} &
								\includegraphics[width=0.235\linewidth]{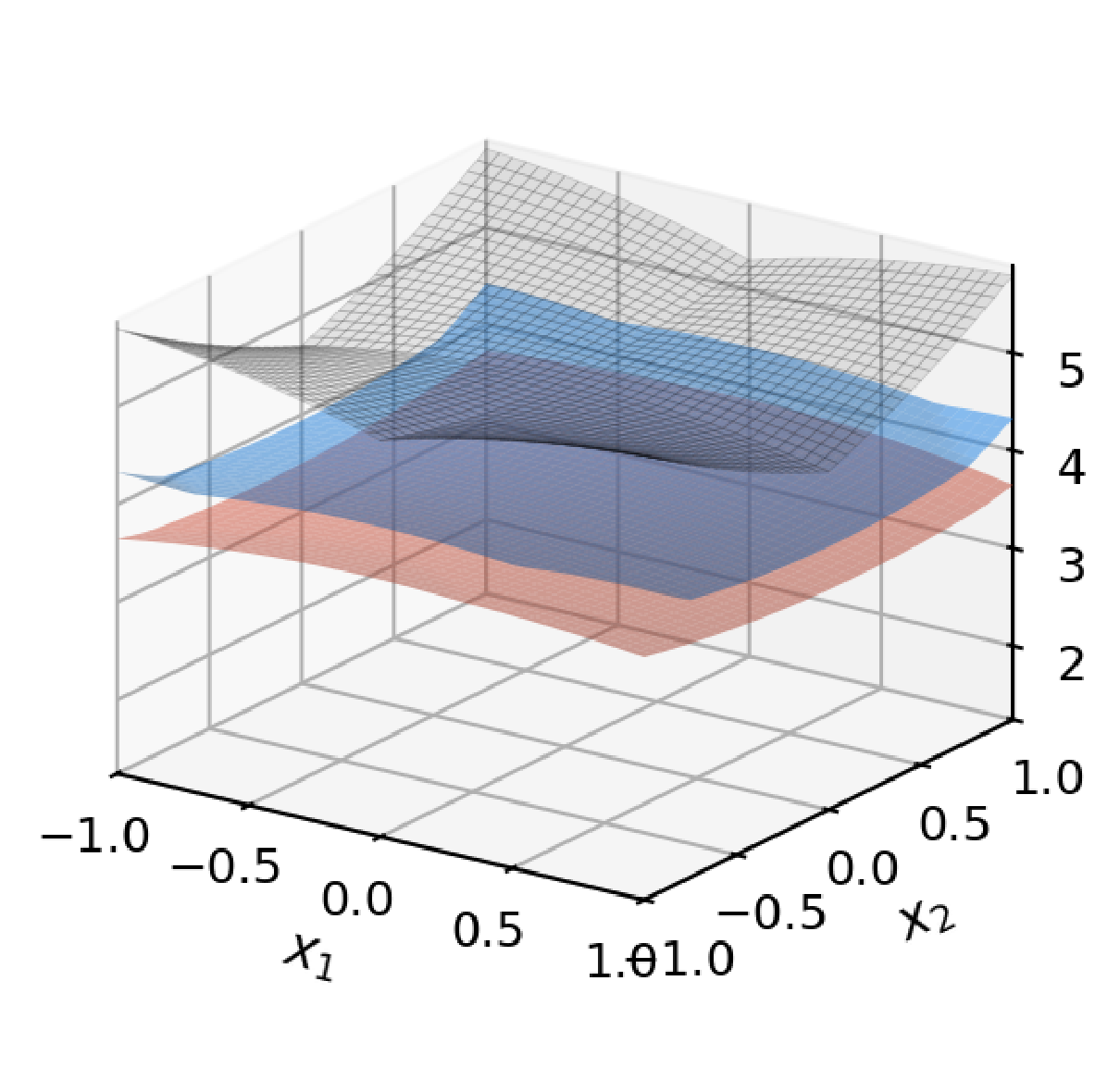} \\
								(a) $ {N}=20$ & (b) $ {N}=50$ & (c) $ {N}=100$ & (d) $ {N}=200$
							\end{tabular}
							\caption{Model surfaces at fixed $m=5$ with increasing sample size $ {N}$. In each panel, the gray surface denotes  $V$, the blue surface denotes  $V_{ {N}}$, and the red surface refers to $\overline V_{ {N}}$.}
							\label{fig:toy-surfaces-fixedm}
						\end{figure}
						
						\paragraph{Approximation errors.}
						Figure~\ref{fig:toy-dimsweepN} reports two diagnostics. In Figure~\ref{fig:toy-dimsweepN}(a), we compare the 
						relative reference value gaps attained by the minimizers of the two sampled models. Specifically,  
						we report
						$
						\frac{v^*-v_N^*}{|v^*|}
						$  for the majorant stochastic approximation and   $
						\frac{v^*-\overline v_N^*}{|v^*|} 
						$  for the standard stochastic approximation.
						Specifically, we use
						$
						m\in\{2,3,5\}, \text{ and }
						N\in\{2,10,20,\ldots,200\},
						$
						and average the results over \(60\) independent sample streams.
						Across $m\in\{2,3,5\}$, the majorant approximation yields substantially smaller relative errors than the standard approximation.
						
						Figure~\ref{fig:toy-dimsweepN}(b) sets \(m=2\) and examines the
						optimal-value gaps as functions of the coverage radius.
						Instead of using a random sample set, we construct a deterministic midpoint
						tensor cover. For
						$
						k\in\{6,8,10,12,16,20,24,32,40\},
						$
						the sample set \(Y_N\) consists of the \(N=k^2\) cell midpoints of the
						uniform \(k\times k\) partition of \([-\pi,\pi]^2\). Its coverage radius is
						therefore
						$
						\beta_N
						=
						\frac{\sqrt{2}}{2}\cdot\frac{2\pi}{k}
						=
						\frac{\sqrt{2}\pi}{\sqrt N}.
						$
						We measure the optimal-value gaps by
						$ v^*-\overline v_N^*,
					  v^*-v_N^*,
					$
						for the standard and majorant stochastic approximations, respectively.
						To quantify the empirical dependence on the coverage radius, we overlay two
						dashed least-squares fitted lines in the log--log scale, i.e.,
						$ a+s\log \beta_N .
						$
						The fitted slopes are \(s=1.24\) for the standard approximation and
						\(s=2.01\) for the majorant approximation. Thus, the observed value gaps
						approximately scale as
					$
						v^*-\overline v_N^*\propto \beta_N^{1.24}, \text{ and }
						v^*-v_N^*\propto \beta_N^{2.01},
					$
						which is consistent with the first- and second-order coverage-radius
						dependence predicted by Theorem~\ref{thm:value solution-gap}.

						\begin{figure}[htbp]
							\centering
							\begin{tabular}{@{}cc@{}}
								\includegraphics[width=0.40\linewidth]{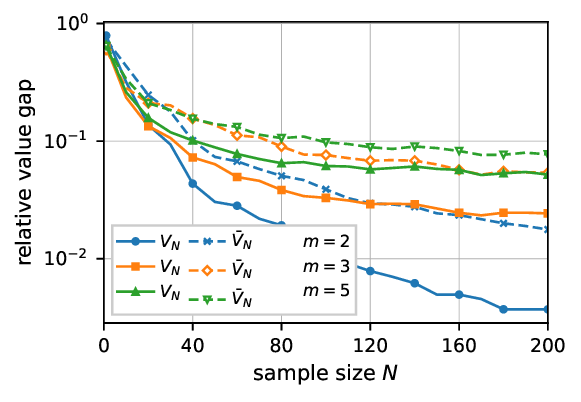} &
								\includegraphics[width=0.40\linewidth]{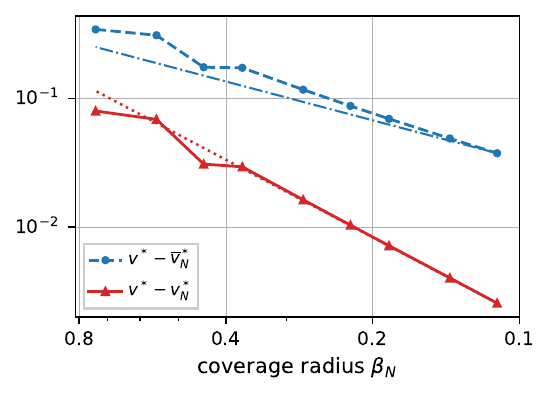} 
							\end{tabular}
							\caption{Approximation errors for the toy example. Panel (a) reports the
								relative optimal-value gaps attained by the minimizers of the sampled models.
								Panel (b) reports the log--log plot of the optimal-value gaps
								\(v^*-\overline v_N^*\) and \(v^*-v_N^*\) versus the coverage radius \(\beta_N\).}
							\label{fig:toy-dimsweepN}
						\end{figure}

						We next evaluate the  SMSA method for solving the majorant stochastic
						approximation model for problem~\eqref{eq:toy_family}. We run
						Algorithm~\ref{alg:growing-N} with   \(\overline{N}=1000\),
						initialized at \(x^{(0)}=(-0.5,0.5)^\top\). The sample size is updated by
						$
						N_0=5, N_{k+1}=\min\{N_k+5,\overline{N}\},
						$
						so that the algorithm reaches the cap after \(200\) outer stages. The parameters are set to  
						$
						\epsilon_k=\max\{10^{-8},10^{-1}(0.9)^k\},
						\mu_k=\frac{\epsilon_k}{2\log(N_k)} .
						$
						The inner projected-gradient loop terminates when
						$
						r_k(z^{(t+1)})\le \epsilon_k$ or  $
						r_k(z^{(t+1)})\le 10^{-4}r_k(x^{(k)}).
						$
						For each $m\in\{2,3,5\}$, we conduct $50$ independent replications. In each replication, a nested family of sample sets is generated from a single i.i.d. sample stream as in~\eqref{eq:ynk}, and Algorithm~\ref{alg:growing-N} is then applied. 
						  Figure~\ref{fig:toy-growing} reports the empirical averages of the smoothed stationarity residual $r_k(  x^{(k)})$ and the approximation error $V( x^{(k)})-v^*$, respectively, where \(x^{(k)}\) denotes the stage-end
						 iterate returned by the \((k-1)\)-th inner solve. Both quantities decrease as $N_k$ grows. 

						\begin{figure}[htbp]
							\centering
							\begin{tabular}{@{}cc@{}}
								\includegraphics[width=0.40\linewidth]{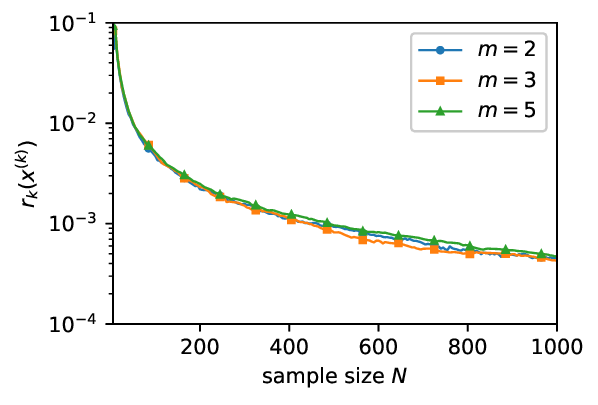} &
								\includegraphics[width=0.40\linewidth]{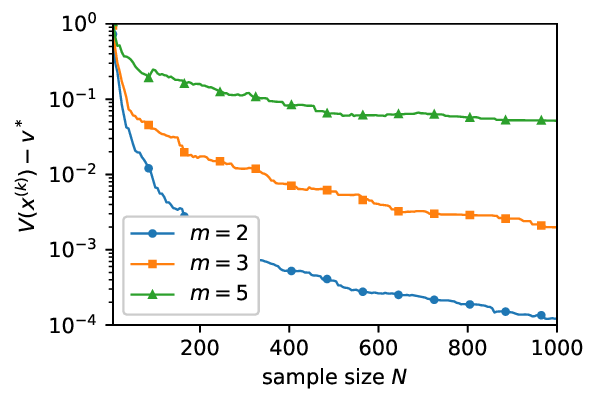} \\
								(a) smoothed stationarity residual & (b) approximation error
							\end{tabular}
							\caption{Performance of the SMSA method with growing samples.}
							\label{fig:toy-growing}
						\end{figure}
						
						\subsection{Robust logistic regression}\label{subsec:rlr}
						We next evaluate the proposed majorant approximation on robust logistic
						regression problems with structured feature perturbations. The experiments
						use two benchmark UCI datasets, Wisconsin Breast Cancer (BC) and Parkinsons
						(PD) \citep{bache2013uci}. The BC dataset contains
						$569$ tumor observations with $d=30$
						real-valued diagnostic features. The PD dataset contains
						$195$ voice-recording observations from 31 subjects,
						with $d=22$ biomedical voice features.
						All reported entries are averaged over the fixed data splits described in
						Appendix~\ref{app:rlr-details}.  
						
						  Let $\{(u_i^{\rm tr},v_i^{\rm tr})\}_{i=1}^{n_{\rm tr}}$
						 denote the training data, where
						 $u_i^{\rm tr}\in\mathbb R^{d}$ is the feature vector and
						 $v_i^{\rm tr}\in\{-1,1\}$ is the class label. We partition the feature
						 coordinates $\{1,\ldots,d\}$ into $m=5$ disjoint groups
						 $\{G_1,\ldots,G_m\}$. For a random perturbation vector
						 $y=(y_1,\ldots,y_m)\in Y_r:=[-r,r]^m$, the perturbed feature
						 $\widetilde u_i^{\rm tr}(y)\in\mathbb R^{d}$ is constructed componentwise by
						 \[
						 \bigl(\widetilde u_i^{\rm tr}(y)\bigr)_k
						 =
						 (1+y_j)(u_i^{\rm tr})_k,
						 \qquad k\in G_j,\quad j=1,\ldots,m .
						 \]
						 Thus, all coordinates in the same group are scaled by the same  
						 perturbation.
						 For  BC dataset,  $
						 \cX=[-2,2]^{30},
						 $
						 and for  PD dataset, $
						 \cX=[-1,1]^{22}.
						 $ The robust training problem
						is 
						\begin{equation}\label{eq:rlr-minimax}
							\min_{x\in \cX}\left\{V(x):=\max_{y\in\cY_{r}} f(x,y)\right\},\quad\text{where}\quad
							f(x,y)=\frac{1}{n_{\rm tr}}\sum_{i=1}^{n_{\rm tr}}\log\left(1+\exp\left(-v_i^{\rm tr}(\widetilde u_i^{\rm tr}(y))^\top x\right)\right).
						\end{equation}
						
						We evaluate the performances using two robust test criteria. Denote
						\(\{(u_i^{\rm te},v_i^{\rm te})\}_{i=1}^{n_{\rm te}}\) as the test set, and \(\widetilde u_i^{\rm te}(y)\)  as the 
					   perturbed test set.  
						The first criterion is the robust test loss
						\begin{equation}\label{eq:rlr-robust-loss}
							L_{\rm rob}(x)
							:=
							\max_{y\in\cY_r}
							\frac{1}{n_{\rm te}}
							\sum_{i=1}^{n_{\rm te}}
							\ell_i^{\rm te}(x,y),
						\end{equation}
						where $
						\ell_i^{\rm te}(x,y)
						:=
						\log\left(
						1+\exp\left(
						-v_i^{\rm te}
						 (\widetilde u_i^{\rm te}(y) )^\top x
						\right)
						\right).
						$
						The second criterion is the conditional value-at-risk (CVaR) of the samplewise
						worst-case losses \citep{rockafellar2000optimization}. For each test sample
						\(i\), define
						$
						L_i^{\rm wc}(x):=
						\max_{y\in\cY_r}\ell_i^{\rm te}(x,y).
						$
						For \(\alpha=0.9\), the CVaR loss is
						\begin{equation}\label{eq:rlr-cvar-loss}
							\operatorname{CVaR}_{\alpha}(x)
							:=
							\min_{\tau\in\RR}
							\left\{
							\tau
							+
							\frac{1}{(1-\alpha)n_{\rm te}}
							\sum_{i=1}^{n_{\rm te}}
							\bigl(L_i^{\rm wc}(x)-\tau\bigr)_+
							\right\}.
						\end{equation}
						This metric summarizes the upper tail of the samplewise adversarial losses.
						The exact numerical evaluation of \(L_{\rm rob}(x)\) and
						\(\operatorname{CVaR}_{0.9}(x)\) over the continuous uncertainty set \(\cY_r\) is described in Appendix~\ref{app:rlr-details}.
						
						We next describe how Algorithm~\ref{alg:growing-N} is instantiated in these
						experiments. For each dataset, data split, and perturbation radius
						$
						r\in\{0.15,0.25,0.35\},
						$
						we first generate 
						$
						Y_{\overline{N}}:=\{y^{(1)},\ldots,y^{(\overline{N})}\},
						$ 
						with the sample size
						$
						\overline{N}\in\{16,64,256,1024\}.
						$
						Here, \(y^{(j)}\sim \mathrm{Unif}(\mathcal Y_r)\) for all $j=1,2,\ldots,\overline{N}$.
						For a row of Table~\ref{tab:rlr-results} with terminal cap \(\overline{N}\), we run
						Algorithm~\ref{alg:growing-N} with 
						$
						x^{(0)}=0,
						$
						\(N_0=2\), and 
						$
						N_k=4^k
						$
						for all $k\ge1$.
						The majorant parameter \(\widehat L\) used in \(q\) is specified in
						Appendix~\ref{app:rlr-details}. At stage \(k\), we set
						\mbox{$
						\epsilon_k^{\rm tol}=10^{-1}(0.9)^k, \text{ and }
						\mu_k=\frac{\epsilon_k^{\rm tol}}{2\log N_k}.
						$}
						Each stage is given a maximum of \(600\) projected-gradient
						updates for  BC dataset and \(500\) projected-gradient updates for  PD dataset.
						To obtain the  ERM reference, we solve
						$
						x_{\rm ERM}
						\in
						\arg\min_{x\in\cX} f(x,0).
						$ 
						For a fair comparison, the standard sampled approximation
					$
						\overline V_{\overline N}(x)
					$
						is solved by the same LSE-smoothing/projected-gradient framework as the
						majorant model, with \(q(x,y_i)\) replaced by \(f(x,y_i)\). Unless stated
						otherwise, we use the same initialization, sample schedule, smoothing schedule,
						tolerance schedule, and stagewise iteration budget for both models.

					 Table~\ref{tab:rlr-results} compares   \(V_{\overline{N}}\),  \(\overline V_{\overline{N}}\), and the ERM reference. Across  \(\overline{N}\in\{16,64,256,1024\}\), perturbation radii, and datasets, the
					majorant approximation attains a lower robust test loss than the standard
					sampled approximation. The improvement is most pronounced when the terminal
					perturbation bank is still relatively small and the perturbation radius is large.
					Over the reported values of \(\overline{N}\), the relative reduction in robust test
					loss, measured against the standard sampled approximation, ranges from \(0.4\%\)
					to \(17.2\%\) on BC and from \(16.3\%\) to \(31.3\%\) on  PD dataset.
					The largest robust-loss gains occur at \(\overline{N}=16\) and \(r=0.35\) on both
					datasets. For example, at \(\overline{N}=16\) on BC, the robust-loss
					reduction increases from \(0.4\%\) at \(r=0.15\) to \(5.7\%\) at \(r=0.25\) and
					\(17.2\%\) at \(r=0.35\); on  PD dataset, the corresponding reductions are
					\(21.2\%\), \(28.8\%\), and \(31.3\%\). For \(r=0.25\) and \(r=0.35\), the gap
					generally narrows as \(\overline{N}\) increases, but the majorant approximation
					continues to dominate the standard sampled approximation in all reported
					settings.
					
					The tail-risk results are consistent with this picture. For the reported values of
					\(\overline{N}\), the reduction in \(\operatorname{CVaR}_{0.9}\), again measured against the
					standard sampled approximation, ranges from \(13.9\%\) to \(32.1\%\) on  BC dataset and from \(4.1\%\) to \(13.1\%\) on   PD dataset. On  BC dataset, the
					CVaR improvement is especially large for \(r=0.25\) and \(r=0.35\), remaining
					above \(26.8\%\) across all reported terminal caps. On  PD dataset, the CVaR
					gains are smaller but positive in every setting, with the largest reduction
					\(13.1\%\) attained at \(r=0.15\) and \(\overline{N}=16\). This improvement is
					meaningful because \(\operatorname{CVaR}_{0.9}\) captures the largest adversarial test
					losses and is therefore sensitive to hard samples near the decision boundary.

						\begin{table*}[t]
							\centering
							\scriptsize
							\setlength{\tabcolsep}{2.4pt}
							\renewcommand{\arraystretch}{1.08}
							\caption{Robust logistic regression results under different perturbation radii.
								For each dataset, we report the robust test loss and $\operatorname{CVaR}_{0.9}$.
								Boldface indicates the better value between
								$V_{\overline{N}}$ and $\overline V_{\overline{N}}$ for the same $(r,\overline{N})$.}
							\label{tab:rlr-results}
							\resizebox{\textwidth}{!}{%
								\begin{tabular}{@{}cc|ccc|ccc|ccc|ccc@{}}
									\toprule
									\multicolumn{2}{c}{\multirow{2}{*}{Setting}}
									& \multicolumn{6}{c}{Breast Cancer}
									& \multicolumn{6}{c}{Parkinsons} \\
									\cmidrule(lr){3-8}\cmidrule(l){9-14}
									\multicolumn{2}{c}{}
									& \multicolumn{3}{c}{Robust loss}
									& \multicolumn{3}{c}{$\operatorname{CVaR}_{0.9}$}
									& \multicolumn{3}{c}{Robust loss}
									& \multicolumn{3}{c}{$\operatorname{CVaR}_{0.9}$} \\
									\cmidrule(lr){1-2}
									\cmidrule(lr){3-5}\cmidrule(lr){6-8}
									\cmidrule(lr){9-11}\cmidrule(l){12-14}
									$r$ & $\overline{N}$
									& ERM & $V_{\overline{N}}$ & $\overline V_{\overline{N}}$
									& ERM & $V_{\overline{N}}$ & $\overline V_{\overline{N}}$
									& ERM & $V_{\overline{N}}$ & $\overline V_{\overline{N}}$
									& ERM & $V_{\overline{N}}$ & $\overline V_{\overline{N}}$ \\
									\midrule
									
									\multirow{4}{*}{$0.15$}
									& $16$
									&  \multirow{4}{*}{0.3881} & \textbf{0.3131} & 0.3145
									& \multirow{4}{*}{3.8638} & \textbf{1.9523} & 2.3008
									& \multirow{4}{*}{0.5995} & \textbf{0.4118} & 0.5223
									& \multirow{4}{*}{2.4956}  & \textbf{1.1154} & 1.2841 \\
									& $64$
									& & \textbf{0.2973} & 0.3016
									& & \textbf{1.9659} & 2.2837
									& & \textbf{0.4116} & 0.5079
									& & \textbf{1.1143} & 1.2534 \\
									& $256$
									& & \textbf{0.2690} & 0.2849
									& & \textbf{1.8485} & 2.2004
									& & \textbf{0.4084} & 0.4939
									& & \textbf{1.1066} & 1.2235 \\
									& $1024$
									& & \textbf{0.2570} & 0.2760
									& & \textbf{1.7972} & 2.1567
									& & \textbf{0.4071} & 0.4864
									& & \textbf{1.1131} & 1.2134 \\
									\midrule
									
									\multirow{4}{*}{$0.25$}
									& $16$
									& \multirow{4}{*}{1.0041}  & \textbf{0.5128} & 0.5437
									& \multirow{4}{*}{6.8129} & \textbf{2.2184} & 3.1990
									& \multirow{4}{*}{0.8139}  & \textbf{0.4542} & 0.6375
									&\multirow{4}{*}{3.2201}   & \textbf{1.3932} & 1.5943 \\
									& $64$
									& & \textbf{0.4739} & 0.4969
									& & \textbf{2.1748} & 3.0890
									& & \textbf{0.4590} & 0.6031
									& & \textbf{1.4140} & 1.5757 \\
									& $256$
									& & \textbf{0.3965} & 0.4070
									& & \textbf{1.9037} & 2.6904
									& & \textbf{0.4406} & 0.5727
									& & \textbf{1.3673} & 1.4769 \\
									& $1024$
									& & \textbf{0.3673} & 0.3738
									& & \textbf{1.8082} & 2.5542
									& & \textbf{0.4381} & 0.5586
									& & \textbf{1.3700} & 1.4476 \\
									\midrule
									
									\multirow{4}{*}{$0.35$}
									& $16$
									& \multirow{4}{*}{1.9810}  & \textbf{0.6905} & 0.8338
									& \multirow{4}{*}{9.8408}  & \textbf{2.7155} & 4.0014
									& \multirow{4}{*}{1.1747} & \textbf{0.4803} & 0.6987
									& \multirow{4}{*}{3.9695} & \textbf{1.5455} & 1.6871 \\
									& $64$
									& & \textbf{0.5919} & 0.6491
									& & \textbf{2.4738} & 3.4103
									& & \textbf{0.4795} & 0.6478
									& & \textbf{1.5876} & 1.6550 \\
									& $256$
									& & \textbf{0.4504} & 0.4758
									& & \textbf{1.9891} & 2.7477
									& & \textbf{0.4474} & 0.5923
									& & \textbf{1.3384} & 1.4884 \\
									& $1024$
									& & \textbf{0.4264} & 0.4354
									& & \textbf{1.8684} & 2.5533
									& & \textbf{0.4557} & 0.5700
									& & \textbf{1.2841} & 1.4412 \\
									\bottomrule
								\end{tabular}%
							}
						\end{table*}

						\section{Conclusion}\label{sec:conclu}

						In this paper, we propose a sequential smoothing majorant stochastic approximation method for solving nonconvex--nonconcave minimax optimization problems. Our method leverages the smoothness structure of the objective function to construct a  majorant stochastic approximation. We show theoretically that this approach significantly improves the approximation accuracy of the standard stochastic approximation   from a first-order rate \(O(\beta_N)\) to a second-order rate \(O(\beta_N^2)\).   We also establish nonasymptotic bounds for optimal values and solution sets, along with subsequential consistency for 
						Clarke stationary points. Numerical experiments   illustrate the advantages of the proposed framework.

						\bibliographystyle{informs2014} 
						\bibliography{sample} 

@misc{bache2013uci,
	title={{UCI} Machine Learning Repository},
	author={K. Bache and M. Lichman},
	year={2013}
}

@article{rockafellar2000optimization,
	title={Optimization of Conditional Value-at-Risk},
	author={R. T. Rockafellar and S. Uryasev},
	journal={The Journal of Risk},
	volume={2},
	pages={21--41},
	year={2000}
}

@article{Pfetsch2023ResilientSystems,
	author = {Pfetsch, Marc E. and Schmitt, Andreas},
	title = {A Generic Optimization Framework for Resilient Systems},
	journal = {Optimization Methods and Software},
	year = {2023},
	volume = {38},
	pages = {356--385}
}

@article{Bodur2022TwoStageLDR,
	author = {Bodur, Merve and Luedtke, James R.},
	title = {Two-stage Linear Decision Rules for Multi-stage Stochastic Programming},
	journal = {Mathematical Programming},
	year = {2022},
	volume = {191},
	pages = {347--380}
}

@article{Rahimian2022DROFrameworks,
	title={Frameworks and results in distributionally robust optimization},
	author={Rahimian, Hamed and Mehrotra, Sanjay},
	journal={Open Journal of Mathematical Optimization},
	volume={3},
	pages={1--85},
	year={2022}
}

@article{Shehadeh2023MobileFacilityDRO,
	author = {Shehadeh, Karmel S.},
	title = {Distributionally Robust Optimization Approaches for a Stochastic Mobile Facility Fleet Sizing, Routing, and Scheduling Problem},
	journal = {Transportation Science},
	year = {2023},
	volume = {57},
	pages = {197--229}, 
}

@article{Gao2023WassersteinDRSO,
	author = {Gao, Rui and Kleywegt, Anton J.},
	title = {Distributionally Robust Stochastic Optimization with {Wasserstein} Distance},
	journal = {Mathematics of Operations Research},
	year = {2023},
	volume = {48},
	pages = {603--655}
}

@article{Gurvich2010CallCenterChance,
	author = {Gurvich, Itai and Luedtke, James and Tezcan, Tolga},
	title = {Staffing Call Centers with Uncertain Demand Forecasts: A Chance-Constrained Optimization Approach},
	journal = {Management Science},
	year = {2010},
	volume = {56},
	pages = {1093--1115}
}

@article{Karimi2021CCLMI,
	author = {Karimi, Roya and Cheng, Jianqiang and Lejeune, Miguel A.},
	title = {A Framework for Solving Chance-Constrained Linear Matrix Inequality Programs},
	journal = {INFORMS Journal on Computing},
	year = {2021},
	volume = {33},
	pages = {1015--1036}
}

@article{Porras2023TightCompactSAA,
	author = {Porras, Alvaro and Dominguez, Concepcion and Morales, Juan Miguel and Pineda, Salvador},
	title = {Tight and Compact Sample Average Approximation for Joint Chance-Constrained Problems with Applications to Optimal Power Flow},
	journal = {INFORMS Journal on Computing},
	year = {2023},
	volume = {35},
	pages = {1454--1469}
}

@article{HoNguyen2023StrongLHS,
	author = {Ho-Nguyen, Nam and Kilinc-Karzan, Fatma and Kucukyavuz, Simge and Lee, Dabeen},
	title = {Strong Formulations for Distributionally Robust Chance-Constrained Programs with Left-Hand Side Uncertainty Under {Wasserstein} Ambiguity},
	journal = {INFORMS Journal on Optimization},
	year = {2023},
	volume = {5},
	pages = {211--232}
}

@article{KilincKarzan2022JointChanceSubmod,
	author = {Kilinc-Karzan, Fatma and Kucukyavuz, Simge and Lee, Dabeen},
	title = {Joint Chance-Constrained Programs and the Intersection of Mixing Sets through a Submodularity Lens},
	journal = {Mathematical Programming},
	year = {2022},
	volume = {195},
	pages = {283--326}
}

@article{LuedtkeAhmed2008SampleApprox,
	author = {Luedtke, James and Ahmed, Shabbir},
	title = {A Sample Approximation Approach for Optimization with Probabilistic Constraints},
	journal = {SIAM Journal on Optimization},
	year = {2008},
	volume = {19},
	pages = {674--699}
}

@article{Rahimian2019EffectiveScenarios,
	author = {Rahimian, Hamed and Bayraksan, Guzin and Homem-de-Mello, Tito},
	title = {Identifying Effective Scenarios in Distributionally Robust Stochastic Programs with Total Variation Distance},
	journal = {Mathematical Programming},
	year = {2019},
	volume = {173},
	pages = {393--430},
}

@misc{goodfellow2014explaining,
	title={Explaining and harnessing adversarial examples},
	author={Goodfellow, Ian J and Shlens, Jonathon and Szegedy, Christian},
	howpublished={Preprint, arXiv:1412.6572},
	year={2014}
}

@misc{rahimian2019distributionally,
	title={Distributionally robust optimization: A review},
	author={H. Rahimian and S. Mehrotra},
	howpublished={Preprint, arXiv:1908.05659},
	year={2019}
}

@article{kuhn2024distributionallyrobustoptimization,
	title={Distributionally robust optimization},
	author={D. Kuhn and S. Shafiee and W. Wiesemann},
	journal={Acta Numerica},
	volume={34},
	pages={579--804},
	year={2025},
	publisher={Cambridge University Press}
}

@inproceedings{
	madry2018towards,
	title={Towards Deep Learning Models Resistant to Adversarial Attacks},
	author={A. Madry and A. Makelov and L. Schmidt and D. Tsipras and A. Vladu},
	booktitle={International Conference on Learning Representations},
	year={2018},
}

@misc{huang2015learning,
	title={Learning with a strong adversary},
	author={R. Huang and B. Xu and D. Schuurmans and C. Szepesv{\'a}ri},
	howpublished={Preprint, arXiv:1511.03034},
	year={2015}
}

@article{Noyan2022DecisionDependentDRO,
	author = {Noyan, Nilay and Rudolf, Gabor and Lejeune, Miguel},
	title = {Distributionally Robust Optimization Under a Decision-Dependent Ambiguity Set with Applications to Machine Scheduling and Humanitarian Logistics},
	journal = {INFORMS Journal on Computing},
	year = {2022},
	volume = {34},
	pages = {729--751}
}

@inproceedings{liao2024error,
	title={Error bounds, {PL} condition, and quadratic growth for weakly convex functions, and linear convergences of proximal point methods},
	author={F. Liao and L. Ding and Y. Zheng},
	booktitle={6th Annual Learning for Dynamics \& Control Conference},
	pages={993--1005},
	year={2024},
	organization={PMLR}
}

@article{iusem2010distances,
	title={Distances between closed convex cones: old and new results},
	author={A. Iusem and A. Seeger},
	journal={Journal of Convex Analysis},
	volume={17},
	pages={1033--1055},
	year={2010}
}

@article{liu2014quantitative,
	title={Quantitative stability analysis of stochastic generalized equations},
	author={Y. Liu and W. Romisch and H. Xu},
	journal={SIAM Journal on Optimization},
	volume={24},
	pages={467--497},
	year={2014},
	publisher={SIAM}
}

@article{jiang2023optimality,
	title={Optimality conditions for nonsmooth nonconvex-nonconcave min-max problems and generative adversarial networks},
	author={J. Jiang and X. Chen},
	journal={SIAM Journal on Mathematics of Data Science},
	volume={5},
	pages={693--722},
	year={2023},
	publisher={SIAM}
}

@article{chen2024minmax,
 	title={Robust solutions of nonlinear least squares problems via min-max optimization},
 	author={Chen, Xiaojun and Kelley, CT},
 	journal={IMA Journal of Numerical Analysis},
 	volume={46},
 	pages={1356--1382},
 	year={2026},
 	publisher={Oxford University Press}
 }

@inproceedings{lin2023gradient,
	title={On gradient descent ascent for nonconvex-concave minimax problems}, 
	author={T. Lin and C. Jin and M. I. Jordan},
	booktitle={International Conference on Machine Learning},
	pages={6083--6093},
	year={2020},
	publisher={PMLR}
}

@article{Neumann1928ZurTD,
	title={Zur theorie der gesellschaftsspiele},
	author={J. von Neumann},
	journal={Mathematische Annalen},
	year={1928},
	volume={100},
	pages={295--320},
}

@article{Nemirovski2004ProxMethodWR,
	title={Prox-method with rate of convergence $O(1/t)$ for variational inequalities with {Lipschitz} continuous monotone operators and smooth convex-concave saddle point problems},
	author={A. Nemirovski},
	journal={SIAM Journal on Optimization},
	year={2004},
	volume={15},
	pages={229--251},
}

@article{Sion1958OnGM,
	title={On general minimax theorems},
	author={Sion, Maurice},
	journal={Pacific Journal of Mathematics},
	volume={8},
	pages={171--176},
	year={1958},
	publisher={Mathematical Sciences Publishers}
}

@article{xu2023unified,
	title={A unified single-loop alternating gradient projection algorithm for nonconvex--concave and convex--nonconcave minimax problems},
	author={Z. Xu and H. Zhang and Y. Xu and G. Lan},
	journal={Mathematical Programming},
	volume={201},
	pages={635--706},
	year={2023},
	publisher={Springer}
}

@article{danskin1966theory,
	title={The theory of max-min, with applications},
	author={J. M. Danskin},
	journal={SIAM Journal on Applied Mathematics},
	volume={14},
	pages={641--664},
	year={1966},
	publisher={SIAM}
}

@book{Berge1963,
	author={C. Berge},
	title={Topological Spaces},
	publisher={Oliver \& Boyd},
	address={Edinburgh},
	year={1963}
}

@article{xu2018distributionally,
	title={Distributionally robust optimization with matrix moment constraints: Lagrange duality and cutting plane methods},
	author={H. Xu and Y. Liu and H. Sun},
	journal={Mathematical Programming},
	volume={169},
	pages={489--529},
	year={2018},
	publisher={Springer}
}

@article{calafiore2005uncertain,
	title={Uncertain convex programs: randomized solutions and confidence levels},
	author={G. Calafiore and M. C. Campi},
	journal={Mathematical Programming},
	volume={102},
	pages={25--46},
	year={2005},
	publisher={Springer}
}

@article{anderson2014confidence,
	title={Confidence levels for {CVaR} risk measures and minimax limits},
	author={E. Anderson and H. Xu and D. Zhang},
	journal={Mathematical Programming},
	volume={180},
	pages={327--370},
	year={2020},
	publisher={Springer}
}

@article{jiang2023pure,
	title={Pure characteristics demand models and distributionally robust mathematical programs with stochastic complementarity constraints},
	author={J. Jiang and X. Chen},
	journal={Mathematical Programming},
	volume={198},
	pages={1449--1484},
	year={2023},
	publisher={Springer}
}

@article{li2022nonsmooth,
	title={Nonsmooth nonconvex–nonconcave minimax optimization: Primal–dual balancing and iteration complexity analysis},
	author={J. Li and L. Zhu and A. M.-C. So},
	journal={Mathematical Programming},
	volume={214},
	pages={591--641},
	year={2025},
	publisher={Springer}
}

@book{clarke1990optimization,
	title={{Optimization and Nonsmooth Analysis}},
	author={F. H. Clarke},
	year={1990},
	publisher={SIAM},
	address={Philadelphia}
}

@article{chen2012smoothing,
	title={Smoothing methods for nonsmooth, nonconvex minimization},
	author={X. Chen},
	journal={Mathematical Programming},
	volume={134},
	pages={71--99},
	year={2012},
	publisher={Springer}
}

@book{shapiro2009lectures,
	author={A. Shapiro and D. Dentcheva and A. Ruszczy{\'n}ski},
	title={Lectures on Stochastic Programming: Modeling and Theory},
	publisher={SIAM},
	address={Philadelphia},
	year={2009}
}

@book{roc1998var,
	title={Variational Analysis},
	author={R. T. Rockafellar and  R. J.-B. Wets},
	year={1998},
	publisher={Springer},
	address = {Berlin, Heidelberg}
}

@book{billingsley1995probability,
	author={P. Billingsley},
	title={Probability and Measure},
	edition={3rd},
	publisher={Wiley},
	address={New York},
	year={1995}
}

@misc{liu2025stochasticsmoothingframeworknonconvexnonconcave,
	title={A stochastic smoothing framework for nonconvex-nonconcave min{E}max problems with applications to {Wasserstein} distributionally robust optimization}, 
	author={W. Liu and M. Khan and G. Mancino-Ball and Y. Xu},
	year={2025},
	howpublished={Preprint, arXiv:2502.17602}
}

@article{xu2024derivativefreeminMax,
	title={Derivative-free alternating projection algorithms for general nonconvex-concave minimax problems}, 
	author={Z. Xu and Z. Wang and J. Shen and Y. Dai},
	journal={SIAM Journal on Optimization},
	year={2024},
	volume={34},
	pages={1879--1908}
}

@article{burke2020subdifferential,
	title={The subdifferential of measurable composite max integrands and smoothing approximation},
	author={J. V. Burke and X. Chen and H. Sun},
	journal={Mathematical Programming},
	volume={181},
	pages={229--264},
	year={2020},
	publisher={Springer}
}

@article{chen2014optimal,
	title={Optimal primal-dual methods for a class of saddle point problems},
	author={Y. Chen and G. Lan and Y. Ouyang},
	journal={SIAM Journal on Optimization},
	volume={24},
	pages={1779--1814},
	year={2014},
	publisher={SIAM}
}

@article{xu2024decentralized,
	title={Decentralized gradient descent maximization method for composite nonconvex strongly-concave minimax problems},
	author={Y. Xu},
	journal={SIAM Journal on Optimization},
	volume={34},
	pages={1006--1044},
	year={2024},
	publisher={SIAM}
}
						
						
						
						
						\newpage
						
						\begin{APPENDICES}
							\AppendixHyperFix
							
							\section{Proofs of the Main Results in Section~\ref{sec:2}}\label{app:A}

							\subsection{Proof of Theorem~\ref{thm:value solution-gap}}\label{appen:1}
							
							\begin{proof}{Proof}
								We first establish a pointwise bound for \(V(x)-V_N(x)\). For a fixed $y\in\cY$, we define the local approximation
								\[
								m_{\widehat{L}}(x,y;z):=
								\langle \nabla_y f(x,y),\,z-y\rangle-\frac{\widehat{L}}{2}\|z-y\|^2.
								\] Fix any \(x\in\cX\), and let
								$
								y^*\in \arg\max_{y\in\cY} f(x,y),
								$
								so that \(V(x)=f(x,y^*)\). By the definition of \(\beta_N\), there exists some \(j\in[N]\) such that
								$
								\|y^*-y_j\|\le \beta_N.
								$
								Since
								$
								V_N(x)=\max_{i\in[N]} q(x,y_i)
								$
								and, by definition of \(q\), 
								we have
								\[
								V_N(x)\ge q(x,y_j)\ge f(x,y_j)+m_{\widehat{L}}(x,y_j;y^*).
								\]
								Therefore,
								\begin{equation*}
									\begin{aligned}
										V(x)-V_N(x)
										&\le f(x,y^*)-f(x,y_j)-m_{\widehat{L}}(x,y_j;y^*) \\
										&= f(x,y^*)-f(x,y_j)-\langle \nabla_y f(x,y_j),\,y^*-y_j\rangle
										+\frac{\widehat{L}}{2}\|y^*-y_j\|^2.
									\end{aligned}
								\end{equation*}
								Using the \(L\)-smoothness of \(f(x,\cdot)\), we obtain
								\[
								f(x,y^*)
								\le
								f(x,y_j)+\langle \nabla_y f(x,y_j),\,y^*-y_j\rangle
								+\frac{L}{2}\|y^*-y_j\|^2.
								\]
								Substituting this estimate into the previous display yields
								\[
								V(x)-V_N(x)
								\le
								\frac{L+\widehat{L}}{2}\|y^*-y_j\|^2
								\le
								\widehat{L}\,\beta_N^2,
								\]
								where the last inequality uses \(\widehat{L}> L\). On the other hand, since \(q(x,y)\le V(x)\) for all
								\((x,y)\in \cX\times\cY\), it follows that \(V_N(x)\le V(x)\). Hence,
								\[
								0\le V(x)-V_N(x)\le \widehat{L}\,\beta_N^2,
								\qquad \forall x\in\cX.
								\]
								Taking the supremum over \(x\in\cX\), we obtain
								\[
								0\le \nu^*-\nu_N^*
								=
								\min_{x\in\cX}V(x)-\min_{x\in\cX}V_N(x)
								\le
								\sup_{x\in\cX}\big(V(x)-V_N(x)\big)
								\le
								\widehat{L}\,\beta_N^2.
								\]
								
								We next prove the second inequality in~\eqref{eq:value-gap-orders-new}. Fix any \(x_N^*\in\cX_N^*\). Since
								\(V_N(x_N^*)=\nu_N^*\) and \(\nu_N^*\le \nu^*\), the pointwise bound above gives
								\[
								V(x_N^*)-\nu^*
								\le
								V(x_N^*)-\nu_N^*
								=
								V(x_N^*)-V_N(x_N^*)
								\le
								\widehat{L}\,\beta_N^2.
								\]
								Moreover, by the definition of \(\cR\) in~\eqref{eq:R-def},
								$
								V(x_N^*)-\nu^*\ge \cR\big(\dist(x_N^*,\cX^*)\big).
								$
								Therefore,
								$
								\cR\big(\dist(x_N^*,\cX^*)\big)\le \widehat{L}\,\beta_N^2,
								$
								which implies
								$
								\dist(x_N^*,\cX^*)\le \cR^{-1}(\widehat{L}\,\beta_N^2).
								$
								Taking the supremum over \(x_N^*\in\cX_N^*\) completes the proof.
								\null\hfill$\square$\end{proof}

							\subsection{Proof of Proposition~\ref{thm:stationary_sets_asymp2}}\label{appen:3}
							\begin{proof}{Proof}
								We present the argument for the  case $x_N\in \cD_N$; the other case is identical after replacing
								$\partial V_N$ with $\partial \overline V_N$.
								For any $N$, since $x_N\in \cD_N$, there exist $g_N\in \partial V_N(x_N)$ and $\xi_N\in \cN_{\cX}(x_N)$ such that
								$g_N+\xi_N=0$. By Assumption~\ref{ass:smoothness} and the compactness of $\cX\times \cY$, we know that $\{x_N\}$ is bounded, $\{g_N = -\xi_N\}$ is bounded, and the normal cone mapping has a closed graph on $\cX$. Therefore,
								we may extract a subsequence (not relabeled) such that $\xi_N\to \xi\in \cN_{\cX}(x)$.
								Moreover, by the subdifferential consistency described above, we can extract a further subsequence along which
								$g_N\to g\in \partial V(x)$. Passing to the limit in $g_N+\xi_N=0$ yields $g+\xi=0$, hence
								$0\in \partial V(x)+\cN_{\cX}(x)$, which means $x\in \cD$.
								\null\hfill$\square$\end{proof}
							
							\subsection{Proofs of Lemma~\ref{lem:QG_f_to_QG_q_global} and Theorem~\ref{thm:main_betaN_global}}\label{appen:4}
							We first establish two technical lemmas used in the proofs of Lemma~\ref{lem:QG_f_to_QG_q_global} and Theorem~\ref{thm:main_betaN_global}.
							\begin{lemma}
								\label{lem:PL_to_QG_f_global}
								Under Assumption
								\ref{ass:PL_f_global}, the following quadratic growth  condition holds:
								\begin{equation}\label{eq:QG_f_global}
									V(x)-f(x,y)\ \ge\ \frac{\kappa_p(x)}{8}\,\dist^2\big(y,S_0(x)\big),
									\qquad \forall x\in\cX, y\in\cY,
								\end{equation}
								where $S_0(x) := \arg\max_{y\in \cY} f(x,y)$.
							\end{lemma}
							\begin{proof}{Proof}
								Define  $\phi_x(y) := -f(x,y)+\iota_{\cY}(y)$. 	
								Consequently, its minimum value is $\phi_x^*=-V(x)$, and the optimal solution set is $\arg\min \phi_x=S_0(x)$. 	Because \(f(x,\cdot)\) is \(L\)-smooth on the convex set \(\cY\), the map \(y\mapsto -f(x,y)\) is \(L\)-weakly convex on \(\cY\). Since \(\iota_{\cY}\) is convex, it follows that \(\phi_x\) is \(L\)-weakly convex. In addition, for every \(y\in\cY\), $	\partial \phi_x(y)= -\nabla_y f(x,y)+\cN_{\cY}(y).
								$ 	Hence Assumption~\ref{ass:PL_f_global} is exactly the global PL inequality for \(\phi_x\).  Invoking the property that the PL condition implies the QG condition \cite[Theorem~3.1]{liao2024error}, we obtain \eqref{eq:QG_f_global} with the modulus $ \kappa_p(x)/8$. This completes the proof.
								\null\hfill$\square$\end{proof}
							
							\begin{lemma}
								\label{lem:explicit_Lxqy_short}
								Assume that $\widehat{L}>L$. 
								Then, for all $y,y'\in\cY$, it holds that
								\begin{equation}\label{eq:subgrad_Lip_explicit_short}
									\left\|\nabla_x q(x,y) - \nabla_x q(x,y')\right\|
									\le
									\left(D_{\cY}L_0 + \left(3+\frac{L}{\widehat{L}}\right)L\right)\|y-y'\|, \text{ and }
								\end{equation}
								\begin{equation}\label{eq:subgrad_Lip_explicit2}
									\left\|\nabla_x q(x,y) - \nabla_x q(x',y)\right\|
									\le
									\left(D_{\cY}L_0 + \left(1+\frac{L}{\widehat{L}}\right)L\right)\|x-x'\|,
								\end{equation}
								where $D_{\cY}$ is the diameter of $\cY$.
							\end{lemma}

							\begin{proof}{Proof}
								For a fixed pair $(x,y)$, the inner maximization problem defining $q(x,y)$ in \eqref{eq:q} is strongly concave with respect to $z$. Consequently, it admits a unique maximizer given by
								$
								z(x,y)
								=
								\Pi_{\cY}\!\left(y+\frac{1}{\widehat{L}}\nabla_y f(x,y)\right).
								$
								Applying Danskin's theorem \citep{danskin1966theory}, the gradient of $q$ with respect to $x$ is
								\begin{equation}\label{eq:grad_q_formula_short}
									\nabla_x q(x,y)
									=
									\nabla_x f(x,y)+\nabla^2_{xy} f(x,y)\big(z(x,y)-y\big).
								\end{equation}
								Let $\Delta(x,y):=z(x,y)-y$ denote the update step. Because the mapping $y\mapsto \nabla_x f(x,y)$ has Jacobian $\nabla^2_{xy} f(x,y)$ and $\cY$ is a compact convex set, the smoothness condition~\eqref{eq:Lipschtiz_smooth} implies that  
								\begin{equation}\label{eq:nablaxf_lip_short}
									\|\nabla_x f(x,y)-\nabla_x f(x,y')\|
									\le
									L\,\|y-y'\|,   \text{ and }   \|\nabla^2_{xy} f(x,y)\| \le L,
									\qquad \forall y,y'\in\cY.
								\end{equation}
								Next, leveraging the nonexpansiveness of the Euclidean projection and the $L$-Lipschitz continuity of $\nabla_y f(x,\cdot)$, we have
								\begin{equation}
									\|z(x,y)-z(x,y')\|
									\le
									\left\|
									y-y'
									+\frac{1}{\widehat{L}}\big(\nabla_y f(x,y)-\nabla_y f(x,y')\big)
									\right\|
									\le
									\left(1+\frac{L}{\widehat{L}}\right)\|y-y'\|.
									\label{eq:z_lip_short}
								\end{equation}
								It then follows from the triangle inequality that
								\begin{equation}\label{eq:delta_lip_short}
									\|\Delta(x,y)-\Delta(x,y')\|
									\le
									\left(2+\frac{L}{\widehat{L}}\right)\|y-y'\|.
								\end{equation}
								Furthermore, because $y, z(x,y) \in \cY$ and $\cY$ has a bounded diameter $D_{\cY}$, we   have
								\begin{equation}\label{eq:delta_bound_short}
									\|\Delta(x,y)\|\le D_{\cY}.
								\end{equation}
								Using  the analytical form \eqref{eq:grad_q_formula_short}, we decompose the gradient difference as
								\begin{equation*}
									\begin{aligned}
										&\nabla_x q(x,y)-\nabla_x q(x,y')
										=
										\big(\nabla_x f(x,y)-\nabla_x f(x,y')\big) \\
										&\quad
										+\big(\nabla^2_{xy} f(x,y)-\nabla^2_{xy} f(x,y')\big)\Delta(x,y)
										+\nabla^2_{xy} f(x,y')\big(\Delta(x,y)-\Delta(x,y')\big).
									\end{aligned}
								\end{equation*}
								Taking the norm on both sides and   applying \eqref{eq:mainLxyy_ass_compact_short}, \eqref{eq:nablaxf_lip_short}, \eqref{eq:delta_lip_short}, and \eqref{eq:delta_bound_short} yields
								\begin{equation*}
									\begin{aligned}
										\|\nabla_x q(x,y)-\nabla_x q(x,y')\|
										&\le
										L \|y-y'\|
										+
										L_0\|\Delta(x,y)\|\,\|y-y'\|
										+
										L\left(2+\frac{L}{\widehat{L}}\right)\|y-y'\| \\
										&\le
										\left(
										D_{\cY}L_0
										+
										\left(3+\frac{L}{\widehat{L}}\right)L
										\right)\|y-y'\|.
									\end{aligned}
								\end{equation*}
								This establishes the desired bound \eqref{eq:subgrad_Lip_explicit_short}.
								
								In addition, we have 
								\begin{equation*}
									\begin{aligned}
										\nabla_x q(x,y)-\nabla_x q(x',y)
										={}&
										\bigl(\nabla_x f(x,y)-\nabla_x f(x',y)\bigr) \\
										&+
										\bigl(\nabla_{xy}^2 f(x,y)-\nabla_{xy}^2 f(x',y)\bigr)\Delta(x,y) +
										\nabla_{xy}^2 f(x',y)\bigl(\Delta(x,y)-\Delta(x',y)\bigr). 
									\end{aligned}
								\end{equation*}
								Using the $L$-smoothness of $f$, \eqref{eq:mainLxyy_ass_compact_short}, and~\eqref{eq:delta_bound_short}, we get
								\[
								\|\nabla_x f(x,y)-\nabla_x f(x',y)\|
								\le
								L\|x-x'\|,
								\]
								\[
								\|\nabla_{xy}^2 f(x,y)-\nabla_{xy}^2 f(x',y)\|\,\|\Delta(x,y)\|
								\le
								D_{\cY}L_0\|x-x'\|, \text{ and } 	\|\nabla_{xy}^2 f(x',y)\|
								\le
								L.
								\]  
								Moreover, by the nonexpansiveness of $\Pi_{\cY}$ and the $L$-Lipschitz continuity of $\nabla_y f(\cdot,y)$,
								\[
								\|\Delta(x,y)-\Delta(x',y)\|
								=
								\|z(x,y)-z(x',y)\|
								\le
								\frac{L}{\widehat{L}}\|x-x'\|.
								\]
								Combining these estimates yields  \eqref{eq:subgrad_Lip_explicit2}. The proof is then completed.
								\null\hfill$\square$\end{proof}
							
							Now, we are ready to prove Lemma~\ref{lem:QG_f_to_QG_q_global} and Theorem~\ref{thm:main_betaN_global} as follows.
							\begin{proof}{Proof of Lemma~\ref{lem:QG_f_to_QG_q_global}.}
								For a fixed $y\in\cY$, we define the local approximation
								\[
								m_{\widehat{L}}(x,y;z):=
								\langle \nabla_y f(x,y),\,z-y\rangle-\frac{\widehat{L}}{2}\|z-y\|^2.
								\]
								By the $L$-smoothness of $f(x,\cdot)$ at $y$, we have for any $z\in\cY$,
								\[
								f(x,z)\ \ge\
								f(x,y)+\langle \nabla_y f(x,y),\,z-y\rangle-\frac{L}{2}\|z-y\|^2.
								\]
								Rearranging the terms yields
								\[
								m_{\widehat{L}}(x,y;z)
								\le
								f(x,z)-f(x,y)-\frac{\widehat{L}-L}{2}\|z-y\|^2.
								\]
								Taking the maximum over $z\in\cY$ and adding $f(x,y)$ to both sides, we obtain the upper bound for $q(x,y)$:
								\[
								q(x,y)
								\le
								\max_{z\in\cY}\Big(f(x,z)-\frac{\widehat{L}-L}{2}\|z-y\|^2\Big).
								\]
								Therefore, the optimality gap for $q(x,y)$ can be lower bounded as
								\begin{equation}\label{eq:V_minus_q_global}
									V(x)-q(x,y)
									\ \ge\
									\min_{z\in\cY}
									\Big(
									V(x)-f(x,z)+\frac{\widehat{L}-L}{2}\|z-y\|^2
									\Big).
								\end{equation}
								Applying Lemma~\ref{lem:PL_to_QG_f_global} to $V(x)-f(x,z)$ gives
								\[
								V(x)-q(x,y)
								\ge
								\min_{z\in\cY}
								\Big(
								\frac{\kappa_p(x)}{8}\,\dist^2\big(z,S_0(x)\big)+\frac{\widehat{L}-L}{2}\|z-y\|^2
								\Big).
								\]
								Let $s\in S_0(x)$ be an arbitrary optimal point for a fixed $x\in\cX$. By the triangle inequality $\|y-s\|\le \|y-z\|+\|z-s\|$ and the  inequality $t_1 a^2+t_2 b^2\ge\frac{t_1t_2}{t_1+t_2}(a+b)^2$ (valid for all $t_1,t_2>0$), we obtain
								\[
								\frac{\kappa_p(x)}{8}\|z-s\|^2+\frac{\widehat{L}-L}{2}\|y-z\|^2 \ge \frac{(\kappa_p(x)/8) \cdot (\widehat{L}-L)/2}{\kappa_p(x)/8+(\widehat{L}-L)/2}\|y-s\|^2.
								\]
								Taking $s=\Pi_{S_0(x)}(z)$  yields
								\begin{equation*}
									\frac{\kappa_p(x)}{8}\,\dist^2\big(z,S_0(x)\big)+\frac{\widehat{L}-L}{2}\|z-y\|^2  \ge \frac{\kappa_p(x)(\widehat{L}-L)}{2\kappa_p(x)+8(\widehat{L}-L)}\dist^2\big(y,S_0(x)\big)  =
									\kappa_q(x)\,\dist^2\big(y,S_0(x)\big).
								\end{equation*}
								Because this lower bound holds uniformly for any $z\in\cY$, we conclude \eqref{eq:QG_q_global}.
								\null\hfill$\square$\end{proof}

							\begin{proof}{Proof of Theorem~\ref{thm:main_betaN_global}.}
								Let $\cJ_N(x):=\{y\in Y_N:\ q(x,y)=V_N(x)\}$.
								Fix an arbitrary $y\in \cJ_N(x)$. By definition, $q(x,y)=V_N(x)$, and invoking Theorem~\ref{thm:value solution-gap} yields the optimality gap:
								\[
								V(x)-q(x,y)=V(x)-V_N(x)\le \widehat{L} \beta_N^2.
								\]
								Applying the quadratic growth property from Lemma~\ref{lem:QG_f_to_QG_q_global} at $y$ gives
								\[
								\kappa_q(x)\,\dist^2\big(y,S_0(x)\big)
								\le
								V(x)-q(x,y)
								\le
								\widehat{L} \beta_N^2,
								\]
								where $S_0(x) := \arg\max_{y\in \cY} f(x,y)$.
								Rearranging this inequality, we obtain
								\begin{equation}\label{eq:active_dist_global}
									\dist\big(y,S_0(x)\big)
									\le
									\sqrt{\frac{\widehat{L} \beta_N^2}{\kappa_q(x)}}.
								\end{equation}
								Since the choice of $y\in \cJ_N(x)$ was arbitrary, taking the supremum over $\cJ_N(x)$ implies
								\begin{equation}\label{eq:active_set_dist_global}
									\sup_{y\in \cJ_N(x)}\dist\big(y,S_0(x)\big)
									\le \sqrt{\frac{\widehat{L}}{\kappa_q(x)}} \beta_N.
								\end{equation}
								
								By Lemma~\ref{lem:QG_f_to_QG_q_global} , any maximizer of $q(x, \cdot)$ must belong to $S_0(x)$. Since every point in $S_0(x)$ also maximizes $q(x, \cdot)$, we have
								$
								\arg \max _{y \in \cY} q(x, y)=S_0(x) .
								$
								Therefore, by Danskin's theorem, it holds that
								\[
								\partial V(x)
								=
								\mathrm{co}\Big(\bigcup_{y\in S_0(x)} \nabla_x q(x,y)\Big),
								\qquad
								\partial V_N(x)
								=
								\mathrm{co}\Big(\bigcup_{y\in \cJ_N(x)} \nabla_x q(x,y)\Big).
								\]Here, $V(x)=\max_{y\in \cY} q(x,y)$ may be attained at multiple points (all of $S_0(x)$).
								Define the corresponding sets of gradients as
								\[
								\mathcal{G}(x):=\bigcup_{y\in S_0(x)} \nabla_x q(x,y)
								\quad \text{and} \quad
								\mathcal{G}_N(x):=\bigcup_{y\in \cJ_N(x)} \nabla_x q(x,y).
								\]
								By the standard property of the   distance between convex hulls \citep[Proposition 1.4]{iusem2010distances}, we have
								\begin{equation}
									\label{eq:vndistance}
									d\big(\partial V_N(x),\partial V(x)\big)
									=
									d\big(\mathrm{co}(\mathcal{G}_N(x)),\mathrm{co}(\mathcal{G}(x))\big)
									\le
									d\big(\mathcal{G}_N(x),\mathcal{G}(x)\big).
								\end{equation}
								Now, consider any gradient $\xi \in \mathcal{G}_N(x)$. By definition, there exists a $y \in \cJ_N(x)$ such that $\xi = \nabla_x q(x,y)$.
								Choose an optimal point $y^*\in S_0(x)$ that achieves the projection distance, such that $\|y-y^*\|=\dist\big(y,S_0(x)\big)$.
								By Lemma~\ref{lem:explicit_Lxqy_short}, the gradient difference is bounded by
								\begin{equation*}
									\begin{aligned}
										\|\xi - \nabla_x q(x,y^*)\| &= \|\nabla_x q(x,y) - \nabla_x q(x,y^*)\| \\
										&\le
										\left(D_{\cY}L_0 + \left(3+\frac{L}{\widehat{L}}\right)L\right)\|y-y^*\|
										=
										\left(D_{\cY}L_0 + \left(3+\frac{L}{\widehat{L}}\right)L\right)\dist\big(y,S_0(x)\big).
									\end{aligned}
								\end{equation*}
								Since $\nabla_x q(x,y^*)\in \mathcal{G}(x)$, we further obtain the point-to-set distance:
								\[
								\dist\big(\xi,\mathcal{G}(x)\big)
								\le
								\left(D_{\cY}L_0 + \left(3+\frac{L}{\widehat{L}}\right)L\right)\dist\big(y,S_0(x)\big).
								\]
								Taking the supremum over all $\xi\in\mathcal{G}_N(x)$ (which corresponds to taking the supremum over $y\in \cJ_N(x)$) gives
								\[
								d\big(\mathcal{G}_N(x),\mathcal{G}(x)\big)
								\le
								\left(D_{\cY}L_0 + \left(3+\frac{L}{\widehat{L}}\right)L\right)\sup_{y\in \cJ_N(x)}\dist\big(y,S_0(x)\big).
								\]
								Substituting the bound \eqref{eq:active_set_dist_global} into this inequality and combining it with \eqref{eq:vndistance} yields the desired bound \eqref{eq:main_bound_beta_global}. This completes the proof.
								\null\hfill$\square$\end{proof}
							
							\subsection{Proof of Proposition~\ref{thm:conv_stat}}\label{appen:5}

							\begin{proof}{Proof}
								Let $x_N \in \mathcal{D}_N$ be an arbitrary stationary point of \eqref{eq:PN}, which implies  $r_N(x_N)=0$.
								For any \(x\in\mathcal X\),
								\[
								r(x)
								=
								\dist\bigl(0,\partial V(x)+\mathcal N_{\mathcal X}(x)\bigr)
								\le
								\dist\bigl(0,\partial V_N(x)+\mathcal N_{\mathcal X}(x)\bigr)
								+
								d\bigl(\partial V_N(x),\partial V(x)\bigr).
								\]
								Applying this inequality at \(x=x_N\in\mathcal D_N\), together with
								\(r_N(x_N)=0\) and Theorem~\ref{thm:main_betaN_global}, yields
								$
								r(x_N)\le L_\kappa(x_N)\beta_N.
								$ 
								Recall the definition of  $\mathcal{R}_{\mathcal{D}}$. By construction, it provides the lower bound
								$r(x_N) \ge
								\mathcal{R}_{\mathcal{D}}\big(\dist(x_N, \mathcal{D})\big).
								$
								Combining these two relations, we arrive at
								$
								\mathcal{R}_{\mathcal{D}}\big(\dist(x_N, \mathcal{D})\big) \le L_{\kappa}(x_N)\beta_N.
								$
								Because $\mathcal{R}_{\mathcal{D}}$ is a growth modulus, its generalized inverse $\mathcal{R}_{\mathcal{D}}^{-1}(t) := \sup\{\tau \ge 0 : \mathcal{R}_{\mathcal{D}}(\tau) \le t\}$ is well-defined and nondecreasing. Applying $\mathcal{R}_{\mathcal{D}}^{-1}$ to both sides of the inequality preserves the order, yielding
								$
								\dist(x_N, \mathcal{D}) \le \mathcal{R}_{\mathcal{D}}^{-1}\big( L_{\kappa}(x_N)\beta_N \big).
								$
								To bound the overall deviation of the set $\mathcal{D}_N$, we take the supremum over all $x_N \in \mathcal{D}_N$ on both sides. Utilizing the nondecreasing property of $\mathcal{R}_{\mathcal{D}}^{-1}$, we conclude that
								\[
								d({\cD}_N,\cD) =\sup_{x_N\in\mathcal{D}_N} \dist(x_N, \mathcal{D}) \le \sup_{x_N\in\mathcal{D}_N} \mathcal{R}_{\mathcal{D}}^{-1}\big( L_{\kappa}(x_N)\beta_N \big)\le  \mathcal{R}_{\mathcal{D}}^{-1}\big( \sup_{x\in\mathcal{D}_N}L_{\kappa}(x)\beta_N \big).
								\]
								This establishes the desired bounds and completes the proof.
								\null\hfill$\square$\end{proof}

							\section{Proofs of the Main Results in Section~\ref{sec:algorithm}}
							
							\subsection{Proof of Lemma~\ref{lem:vi-lipschitz}}\label{appen:B1}

							\begin{proof}{Proof} 
								Fix any \(i\in [N]\) and \(x_1,x_2\in\cX\). For \(r=1,2\), let
								\[
								y_i^*(x_r)\in \argmax_{z\in\cY} m_{\widehat{L}}(x_r,y_i;z).
								\]
								Then
								\[
								q(x_r,y_i)=f(x_r,y_i)+m_{\widehat{L}}\bigl(x_r,y_i;y_i^*(x_r)\bigr),
								\qquad r=1,2.
								\]
								Since \(y_i^*(x_1)\) is feasible for the maximization defining \(q(x_2,y_i)\), we have
								\begin{equation*}
									\begin{aligned}
										q(x_1,y_i)-q(x_2,y_i)
										&\le f(x_1,y_i)+m_{\widehat{L}}\bigl(x_1,y_i;y_i^*(x_1)\bigr)
										-f(x_2,y_i)-m_{\widehat{L}}\bigl(x_2,y_i;y_i^*(x_1)\bigr) \\
										&= \bigl(f(x_1,y_i)-f(x_2,y_i)\bigr)
										+ \Big\langle \nabla_y f(x_1,y_i)-\nabla_y f(x_2,y_i),\, y_i^*(x_1)-y_i \Big\rangle .
									\end{aligned}
								\end{equation*}
								Using the \(L\)-Lipschitz continuity of $f$,  the \(L\)-Lipschitz continuity of \(\nabla_y f\) in~\eqref{eq:Lipschtiz_smooth}, and the fact that
								\(y_i^*(x_1),y_i\in\cY\), so \(\|y_i^*(x_1)-y_i\|\le D_{\cY}\), we obtain
								\[
								q(x_1,y_i)-q(x_2,y_i)
								\le L\|x_1-x_2\|+L D_{\cY}\|x_1-x_2\|
								= G_v\|x_1-x_2\|.
								\]
								Interchanging \(x_1\) and \(x_2\) yields the reverse bound, and hence
								\[
								|q(x_1,y_i)-q(x_2,y_i)|\le G_v\|x_1-x_2\|.
								\]
								Therefore, the mapping \(x\mapsto q(x,y_i)\) is \(G_v\)-Lipschitz on \(\cX\). Since it is differentiable on the compact set \(\cX\), we further obtain
								$
								\|\nabla_x q(x,y_i)\|\le G_v \text{ for all } x\in\cX.
								$
								\null\hfill$\square$\end{proof}
							
							\subsection{Proof of Lemma~\ref{thm:inner_complexity}}\label{appen:B2}
							
							\begin{proof}{Proof}
								(a)
								By the projection step~\eqref{eq:update-3.15}, there exists $\xi^{(t+1)}\in {\mathcal N}_{\mathcal X}(z^{(t+1)})$ such that
								\begin{equation}\label{eq:inner_opt_cond}
									\frac{1}{\alpha_k}(z^{(t+1)}-z^{(t)}) + \nabla_z\widetilde V_{N_k}(z^{(t)},\mu_k) + \xi^{(t+1)} = 0.
								\end{equation}
								Taking the inner product with $z^{(t+1)}-z^{(t)}$ and using $\langle \xi^{(t+1)}, z^{(t)}-z^{(t+1)}\rangle \leq 0$, we obtain
								\begin{equation*}
									\langle \nabla_z\widetilde V_{N_k}(z^{(t)},\mu_k), z^{(t+1)}-z^{(t)}\rangle \leq -\frac{1}{\alpha_k} \|z^{(t+1)}-z^{(t)}\|^2 = -L_{\mu_k} \|z^{(t+1)}-z^{(t)}\|^2.
								\end{equation*}
								Since \(\widetilde V_{N_k}(\cdot,\mu_k)\) is \(L_{\mu_k}\)-smooth, we have
								\[
								\widetilde V_{N_k}(z^{(t+1)},\mu_k)
								\le
								\widetilde V_{N_k}(z^{(t)},\mu_k)
								+\Big\langle \nabla_z\widetilde V_{N_k}(z^{(t)},\mu_k),\, z^{(t+1)}-z^{(t)}\Big\rangle
								+\frac{L_{\mu_k}}{2}\|z^{(t+1)}-z^{(t)}\|^2.
								\]
								Combining the above two inequalities, we have~\eqref{eq:inner_descent}.

								(b) We next bound the stationarity residual $r_k(z^{(t+1)})$. By~\eqref{eq:inner_opt_cond},
								\[
								\nabla_z\widetilde V_{N_k}(z^{(t+1)},\mu_k)+\xi^{(t+1)}
								=
								\Big(\nabla_z\widetilde V_{N_k}(z^{(t+1)},\mu_k)-\nabla_z\widetilde V_{N_k}(z^{(t)},\mu_k)\Big)
								-\frac{1}{\alpha_k}\left(z^{(t+1)}-z^{(t)}\right).
								\]
								Applying the triangle inequality, we obtain
								\begin{equation*}
									\begin{aligned}
										r_{k}(z^{(t+1)}) &\leq \big\|\nabla_z\widetilde V_{N_k}(z^{(t+1)},\mu_k) - \nabla_z\widetilde V_{N_k}(z^{(t)},\mu_k)\big\| + \frac{1}{\alpha_k}\|z^{(t+1)}-z^{(t)}\| \\
										&\leq L_{\mu_k}\|z^{(t+1)}-z^{(t)}\| + \frac{1}{\alpha_k}\|z^{(t+1)}-z^{(t)}\| = 2L_{\mu_k}\|z^{(t+1)}-z^{(t)}\|,
									\end{aligned}
								\end{equation*}
								Combining this estimate with~\eqref{eq:inner_descent} yields
								\begin{equation}\label{eq:inner_residual}
									(r_{k}(z^{(t+1)}))^2 \leq 8 L_{\mu_k} \Big(\widetilde V_{N_k}(z^{(t)},\mu_k) - \widetilde V_{N_k}(z^{(t+1)},\mu_k)\Big).
								\end{equation}
								Now suppose that the inner loop runs for \(T\) iterations. Summing~\eqref{eq:inner_residual} from \(t=0\) to \(T-1\) gives
								\begin{equation*}
									\begin{aligned}
										\sum_{t=0}^{T-1} (r_{k}(z^{(t+1)}))^2 & \leq 8 L_{\mu_k} \Big(\widetilde V_{N_k}(z^{(0)},\mu_k) - \widetilde V_{N_k}(z^{(T)},\mu_k)\Big) \\
										& \leq 8 L_{\mu_k} \Big(\widetilde V_{N_k}(x^{(k)},\mu_k) - \min_{z \in \cX} \widetilde V_{N_k}(z, \mu_k)\Big).
									\end{aligned}
								\end{equation*}
								Hence
								\begin{equation*}
									\min_{0 \le t \le T-1} (r_{k}(z^{(t+1)}))^2 \leq \frac{1}{T} \sum_{t=0}^{T-1} (r_{k}(z^{(t+1)}))^2 \le \frac{8 L_{\mu_k}}{T} \Big(\widetilde V_{N_k}(x^{(k)},\mu_k) - \min_{z \in \cX} \widetilde V_{N_k}(z, \mu_k)\Big).
								\end{equation*}
								Since \(\mu_k\le \epsilon_k<1\), we have
								$
								L_{\mu_k}
								\le \frac{C}{\mu_k}.
								$
								Therefore, the condition
								\[
								\frac{8C}{T\mu_k}
								\Big(
								\widetilde V_{N_k}(x^{(k)},\mu_k)-\min_{z\in\cX}\widetilde V_{N_k}(z,\mu_k)
								\Big)
								\le \epsilon_k^2
								\]
								guarantees that
								$\min_{0 \le t \le T-1} r_{k}(z^{(t+1)}) \le \epsilon_k$.
								This completes the proof by letting \mbox{$T=T_k$}.
								\null\hfill$\square$\end{proof}
							
							\subsection{Proof of Lemma~\ref{lem:lse-grad-to-surrogate}}\label{appen:B3}
							\begin{proof}{Proof}
								By Lemma~\ref{lem:lse-basic}(b),
								\[
								\nabla_x\widetilde V_N(x,\mu)=\sum_{i=1}^N p_i(x,\mu)\nabla_x q(x,y_i),
								\]
								where \(\{p_i(x,\mu)\}_{i=1}^N\) denote the softmax weights. To construct a valid \(\epsilon\)-subgradient of \(V_N(x)\), we isolate the components associated with the \(\epsilon\)-active set, denoted by \(I_{N,\epsilon}(x)\). Here, the set $I_{N, \epsilon}(x)$ is nonempty because it contains every maximizer of $V_N(x)$.
								By normalizing the weights strictly over this active set, we define
								\[
								\widehat p_i:=\frac{p_i(x,\mu)}{\sum_{j\in I_{N,\epsilon}(x)}p_j(x,\mu)},\quad i\in I_{N,\epsilon}(x),
								\]
								and  
								\[
								g_\epsilon:=\sum_{i\in I_{N,\epsilon}(x)}\widehat p_i\,\nabla_x q(x,y_i)\in \partial_\epsilon V_N(x).
								\]
								Using \(\|\nabla_x q(x,y_i)\|\le G_v\) from Lemma~\ref{lem:vi-lipschitz}, we obtain
								\begin{equation*}
									\begin{aligned}
										\operatorname{dist}\big(\nabla_x\widetilde V_N(x,\mu), \partial_\epsilon V_N(x)\big)
										&\le \|\nabla_x\widetilde V_N(x,\mu)-g_\epsilon\|\le 2G_v\sum_{i\notin I_{N,\epsilon}(x)}p_i(x,\mu).
									\end{aligned}
								\end{equation*}
								For any inactive index \(i \notin I_{N,\epsilon}(x)\), it holds by definition that \(q(x,y_i) < V_N(x) - \epsilon\). Consequently,
								\[
								p_i(x,\mu)
								=
								\frac{\exp(q(x,y_i)/\mu)}{\sum_{j=1}^N \exp(q(x,y_j)/\mu)}
								\le
								\frac{\exp((V_N(x)-\epsilon)/\mu)}{\exp(V_N(x)/\mu)}
								=
								\exp\Big(-\frac{\epsilon}{\mu}\Big).
								\]
								Summing this exponential bound over the at most \(N\) inactive indices yields the desired result in~\eqref{eq:lse-grad-to-surrogate}.
								\null\hfill$\square$\end{proof}
							
							\subsection{Proof of Theorem~\ref{thm:growing_N} and the results in Remark~\ref{rem:limit_eps_stationarity}}\label{appen:B4}
							
							\begin{proof}{Proof of Theorem~\ref{thm:growing_N}.}
								Let
							$
							\overline	\epsilon:=\epsilon^{(K)}, $ and $
 							\overline	\mu :=\mu^{(K)} .
							$ By the inner-loop termination criterion of Algorithm~\ref{alg:growing-N}, it holds
								\[
								\operatorname{dist}\!\left(
								0,\,
								\nabla_x\widetilde V_{\overline N}(\bar x,\overline \mu)
								+\cN_{\cX}(\bar x)
								\right)
								\le \overline\epsilon.
								\]
								Hence there exists \(\overline\xi\in\cN_{\cX}(\bar x)\) such that
								\begin{equation}\label{eq:B4-terminal-smoothed-residual}
									\bigl\|
									\nabla_x\widetilde V_{\overline N}(\bar x,\overline\mu)+\overline\xi
									\bigr\|
									\le \overline\epsilon .
								\end{equation}
								
								We now compare the smoothed gradient with the
								\(\overline\epsilon\)-approximate subdifferential of \(V_{\overline N}\).
								By Lemma~\ref{lem:lse-grad-to-surrogate}, applied with
								\(N=\overline N\), \(x=\bar x\), \(\mu=\overline\mu\), and
								\(\epsilon=\overline\epsilon\), we have
								$
								\operatorname{dist}\!\left(
								\nabla_x\widetilde V_{\overline N}(\bar x,\overline\mu),
								\partial_{\overline\epsilon}V_{\overline N}(\bar x)
								\right)
								\le
								2G_v\overline N
								\exp (-\frac{\overline\epsilon}{\overline\mu} ).
								$
								Since
								$
								\overline\mu\le \frac{\overline\epsilon}{2\log(\overline N)},
								$
								it follows that 
								\begin{equation}\label{eq:B4-lse-error-finite}
									\operatorname{dist}\!\left(
									\nabla_x\widetilde V_{\overline N}(\bar x,\overline \mu ),
									\partial_{\overline \epsilon }V_{\overline N}(\bar x)
									\right)
									\le
									\frac{2G_v}{\overline N}
									\le
									\frac{\epsilon}{2},
								\end{equation}
								where the last inequality uses
								\(\overline N\ge 4G_v/\epsilon\).
								
								By~\eqref{eq:B4-lse-error-finite}, there exists
								$
								\overline g \in \partial_{\overline \epsilon }V_{\overline N}(\bar x)
								$
								such that
								$
								\bigl\|
								\overline g -\nabla_x\widetilde V_{\overline N}(\bar x,\overline\mu )
								\bigr\|
								\le \frac{\epsilon}{2}.
								$
								Combining this estimate with~\eqref{eq:B4-terminal-smoothed-residual}
								gives
								\[
								\|\overline g +\overline \xi \|
								\le
								\bigl\|
								\overline g-\nabla_x\widetilde V_{\overline N}(\bar x,\overline \mu )
								\bigr\|
								+
								\bigl\|
								\nabla_x\widetilde V_{\overline N}(\bar x,\overline \mu )+\overline\xi
								\bigr\| 
								\le \epsilon .
								\]
								Moreover, since \(\overline \epsilon \le \epsilon\), we have
							$
								I_{\overline N,\overline \epsilon }(\bar x)
								\subseteq
								I_{\overline N,\epsilon}(\bar x),
						$
								and hence
							$
								\partial_{\overline \epsilon }V_{\overline N}(\bar x)
								\subseteq
								\partial_{\epsilon}V_{\overline N}(\bar x).
							$
								Thus \(\overline g \in\partial_\epsilon V_{\overline N}(\bar x)\), while
								\(\overline \xi \in\cN_{\cX}(\bar x)\). Consequently,
								\[
								\operatorname{dist}\!\left(
								0,\,
								\partial_{\epsilon}V_{\overline N}(\bar x)
								+\cN_{\cX}(\bar x)
								\right)
								\le
								\|\overline g +\overline \xi \|
								\le \epsilon .
								\]
								This completes the proof.
								\null\hfill$\square$\end{proof}
								
								\begin{proof}{Proof of the results in Remark~\ref{rem:limit_eps_stationarity}.}
							 By
								Assumption~\ref{ass:sample}, Lemma~\ref{lem:beta_N_key}, and the
								Borel--Cantelli lemma, along a probability-one sample path we have 
								\(\overline N_\ell\to\infty\),
							$
								\beta_{\overline N_\ell}\to0.
							$ We fix such a sample path throughout the proof.
								
								Let \(x^*\) be an arbitrary accumulation point of
								\(\{\bar x^\ell\}_{\ell\ge1}\). Passing to a subsequence if necessary, and
								not relabeling it, we may assume that
							 $
								\bar x^\ell\to x^* .
							$
								By the assumed approximate stationarity condition, for each \(\ell\) there
								exist
							$
								g^\ell\in
								\partial_{\eta_\ell}V_{\overline N_\ell}(\bar x^\ell),
								$ ang $
								\xi^\ell\in \cN_{\cX}(\bar x^\ell),
							$
								such that
								\begin{equation}\label{eq:B4-rem-approx-stationarity}
									\|g^\ell+\xi^\ell\|\le \eta_\ell .
								\end{equation}
								Indeed, \(\partial_{\eta_\ell}V_{\overline N_\ell}(\bar x^\ell)\) is a
								compact convex set and \(\cN_{\cX}(\bar x^\ell)\) is closed and convex, so
								the distance is attained.
								
								By the definition of
								\(\partial_{\eta_\ell}V_{\overline N_\ell}(\bar x^\ell)\) and
								Lemma~\ref{lem:vi-lipschitz}, the vectors \(g^\ell\) are uniformly bounded:
								by $G_v$.
								Together with~\eqref{eq:B4-rem-approx-stationarity} and
								\(\eta_\ell\to0\), this implies that \(\{\xi^\ell\}_{\ell\ge1}\) is also
								bounded. Passing to a further subsequence if necessary, we may assume that
								$
								g^\ell\to g^*, $ and $
								\xi^\ell\to \xi^* .
							$
								Taking the limit in~\eqref{eq:B4-rem-approx-stationarity} gives
							$
								g^*+\xi^*=0 .
							$
								Moreover, since \(\bar x^\ell\to x^*\), \(\xi^\ell\to\xi^*\), and
								\(\xi^\ell\in\cN_{\cX}(\bar x^\ell)\), the closed-graph property of the
								normal-cone mapping for the closed convex set \(\cX\) yields
								$
								\xi^*\in\cN_{\cX}(x^*) .
							$
								
								It remains to prove that \(g^*\in\partial V(x^*)\). By Carath\'eodory's
								theorem in \(\mathbb R^n\), for each \(\ell\), there exist coefficients
							$
								\lambda_m^\ell\ge0,  $ $
								\sum_{m=1}^{n+1}\lambda_m^\ell=1,$
								$  
								$
								and sample points
								$
								y_m^\ell\in Y_{\overline N_\ell},
								$ for all $ m=1,\ldots,n+1,
								$
								such that
								\begin{equation}\label{eq:B4-rem-caratheodory}
									g^\ell
									=
									\sum_{m=1}^{n+1}
									\lambda_m^\ell
									\nabla_x q(\bar x^\ell,y_m^\ell),
								\end{equation}
								and
								\begin{equation}\label{eq:B4-rem-active}
									q(\bar x^\ell,y_m^\ell)
									\ge
									V_{\overline N_\ell}(\bar x^\ell)-\eta_\ell,
									\qquad
									\forall m=1,\ldots,n+1 .
								\end{equation}
								Since \(\cY\) and the unit simplex in \(\mathbb R^{n+1}\) are compact, we
								may pass to a further subsequence such that, for each \(m=1,\ldots,n+1\),
								$
								y_m^\ell\to y_m^*\in\cY,
								\lambda_m^\ell\to\lambda_m^* .
							$ 
								By Theorem~\ref{thm:value solution-gap}, for all \(x\in\cX\),
								$
								0\le V(x)-V_{\overline N_\ell}(x)
								\le \widehat L\,\beta_{\overline N_\ell}^2 .
							$
								Since \(\beta_{\overline N_\ell}\to0\), it follows that
								\begin{equation}\label{eq:B4-rem-uniform-convergence}
									\sup_{x\in\cX}
									|V_{\overline N_\ell}(x)-V(x)|
									\to0 .
								\end{equation}
								Combining~\eqref{eq:B4-rem-uniform-convergence} with
								\(\bar x^\ell\to x^*\) and the continuity of \(V\), we obtain
							$
								V_{\overline N_\ell}(\bar x^\ell)\to V(x^*) .
							$
								Passing to the limit in~\eqref{eq:B4-rem-active}, using
								\(\eta_\ell\to0\) and the continuity of \(q\), gives
							$
								q(x^*,y_m^*)\ge V(x^*),\forall
								m=1,\ldots,n+1 .
							$
								Since
							$
								V(x^*)=\max_{y\in\cY} q(x^*,y),
							$
								we also have \(q(x^*,y_m^*)\le V(x^*)\). Therefore,
							$
								y_m^*\in \argmax_{y\in\cY} q(x^*,y),\forall
								m=1,\ldots,n+1 .
							$
								
								Assumption~\ref{ass:smoothness} and the continuity of the Euclidean
								projection imply that \(\nabla_x q\) is continuous on \(\cX\times\cY\).
								Taking the limit in~\eqref{eq:B4-rem-caratheodory}, we obtain
								\[
								g^*
								=
								\sum_{m=1}^{n+1}
								\lambda_m^*
								\nabla_x q(x^*,y_m^*)
								\in
								\operatorname{co}
								\left\{
								\nabla_x q(x^*,y):
								y\in\argmax_{y\in\cY}q(x^*,y)
								\right\}.
								\]
								Applying Danskin's theorem to
								$
								V(x)=\max_{y\in\cY}q(x,y),
							$
								we get
							$
								g^*\in\partial V(x^*) .
							$
								Combining this with 
								$
								\xi^*\in\cN_{\cX}(x^*),
							$ and $
								g^*+\xi^*=0,
							$
								we conclude that
								$
								0\in \partial V(x^*)+\cN_{\cX}(x^*) .
							$
								Since \(x^*\) was an arbitrary accumulation point, every accumulation point
								of \(\{\bar x^\ell\}_{\ell\ge1}\) is a Clarke stationary point of
								problem~\eqref{eq:value_problem} with probability one.
					\null\hfill$\square$\end{proof}

							\section{Additional Numerical Experiment Details}\label{app:numerics}
							
							This appendix provides additional implementation details for the numerical
							experiments in Section~\ref{sec:numer}. 
							
							\subsection{Details for the toy minimax example}\label{app:toy-details}
							
							For the toy problem~\eqref{eq:toy_family}, we write \(V_m\) when the dependence
							on the inner dimension \(m\) needs to be emphasized. Let
							\[
							n_o:=\left\lceil\frac{m}{2}\right\rceil,\qquad
							n_e:=\left\lfloor\frac{m}{2}\right\rfloor, \qquad a_o(x):=\frac{x_1^2-x_2^2}{5},\text{ and }
							a_e(x):=\frac{2x_1x_2}{5}.
							\] 
							The 
							value function can be evaluated as
							\[
							V(x)
							=
							n_o\,\phi_{\sin}(a_o(x))
							+
							n_e\,\phi_{\cos}(a_e(x)),
							\]
							where
							$
							\phi_{\sin}(a)
							:=
							\max_{t\in[-\pi,\pi]}\{at-\sin t\},
							\phi_{\cos}(b)
							:=
							\max_{t\in[-\pi,\pi]}\{bt-\cos t\}.
							$
							For \(x\in[-1,1]^2\), we have \(|a_o(x)|\le 1/5\) and
							\(|a_e(x)|\le 2/5\). Hence the above    problems
							admit the closed forms
							$
							\phi_{\sin}(a)
							=
							\sqrt{1-a^2}-a\arccos(a),
							\forall |a|\le 1,
							$
							and
							$
							\phi_{\cos}(b)
							=
							1+\pi |b|,
							\forall |b|\le 1.
							$
							Consequently, \(V(x)\) is evaluated directly from the above expression, rather
							than by solving a nonconvex inner maximization problem. The corresponding
							optimal value is
							\[
							v^*
							=
							n_o
							\left[
							\sqrt{1-\frac{1}{25}}
							-
							\frac{1}{5}\arccos\left(\frac{1}{5}\right)
							\right]
							+
							n_e,
							\]
							which is attained, for example, at \(x=(1,0)^\top\) and
							\(x=(-1,0)^\top\).
							
							For the majorant approximation \(V_N(x)=\max_{i\in[N]}q_m(x,y_i)\), it holds
							$
							\nabla_{yy}^2 F_m(x,y)
							=
							S_o\Diag(\sin y)+S_e\Diag(\cos y),
							$
							and therefore
							$
							\|\nabla_{yy}^2 F_m(x,y)\|\le 1,
							$ for all $ (x,y)\in[-1,1]^2\times[-\pi,\pi]^m .
							$
							Hence,  
							$
							\widehat L=1.5
							$
							is large than 1.

							For Figure~\ref{fig:toy-dimsweepN}(a), the reference grid in \(x\) is
							$
							\mathcal G_x^{61}
							:=
							\left\{
							-1+\frac{2r}{60}: r=0,\ldots,60
							\right\}^2 .
							$
							The reported curves average over \(60\) independent sample streams, generated
							from three base seeds and \(20\) replications per base seed.

							\subsection{Details for robust logistic regression}\label{app:rlr-details}
							
							We now describe the additional implementation details for the robust logistic regression
							experiments in Section~\ref{subsec:rlr}. All experiments use \(m=5\) feature
							groups and perturbation radii
							$
							r\in\{0.15,0.25,0.35\}.
							$
							For each dataset, labels are encoded as \(\{-1,1\}\). The implementation
							standardizes features using the training set only.  
							
							For Breast Cancer, we use samplewise stratified train/validation/test splits
							with ratio \(60\%/20\%/20\%\) and split seeds
							$
							202600,\ 202601,\ 202602,\ 202603,\ 202604.
							$
							The five feature groups are formed by feature-name keywords:
							\[\scriptsize
							\begin{aligned}
								G_1&:\ \{\texttt{radius},\texttt{perimeter},\texttt{area}\},\\
								G_2&:\ \{\texttt{texture}\},\\
								G_3&:\ \{\texttt{smoothness}\},\\
								G_4&:\ \{\texttt{compactness},\texttt{concavity},
								\texttt{concave points}\},\\
								G_5&:\ \{\texttt{symmetry},\texttt{fractal dimension}\}.
							\end{aligned}
							\]
							For Parkinsons, we use subject-level stratified train/validation/test splits
							with the same \(60\%/20\%/20\%\) ratio and split seeds
							$
							20000517,\ 20000518,\ 20000617,\ 20000618.
							$
							The five feature groups are
							\[
							\scriptsize
							\begin{aligned}
								G_1&:\ \{\texttt{MDVP:Fo(Hz)},\texttt{MDVP:Fhi(Hz)},
								\texttt{MDVP:Flo(Hz)}\},\\
								G_2&:\ \{\texttt{MDVP:Jitter(\%)},\texttt{MDVP:Jitter(Abs)},
								\texttt{MDVP:RAP},\texttt{MDVP:PPQ},\texttt{Jitter:DDP}\},\\
								G_3&:\ \{\texttt{MDVP:Shimmer},\texttt{Shimmer(dB)},
								\texttt{MDVP:Shimmer(dB)},\texttt{Shimmer:APQ3},
								\texttt{Shimmer:APQ5},\texttt{MDVP:APQ},
								\texttt{Shimmer:DDA}\},\\
								G_4&:\ \{\texttt{NHR},\texttt{HNR}\},\\
								G_5&:\ \{\texttt{RPDE},\texttt{D2},\texttt{DFA},
								\texttt{spread1},\texttt{spread2},\texttt{PPE}\}.
							\end{aligned}
							\]

							For the majorant approximation, the parameter \(\widehat L\) is computed
							separately for each dataset, perturbation radius, and train/validation/test
							split.  
							Let \(G\in\{0,1\}^{d\times 5}\) be the feature-group assignment matrix, where
							\(G_{rj}=1\) if feature coordinate \(r\) belongs to group \(j\), and
							\(G_{rj}=0\) otherwise. Let \(\sigma_{\rm tr}\in\RR^d\) denote the coordinatewise
							standard deviation of the unperturbed training features, computed from the
							training set only. For \(i=1,\ldots,n_{\rm tr}\) and \(j=1,\ldots,5\), define a matrix $A\in\RR^{n_{\rm tr}\times 5}$ with
							$
							A_{ij}
							:=
							2
							\sum_{r=1}^d
							\frac{(u_i^{\rm tr})_r}{(\sigma_{\rm tr})_r}G_{rj}.
							$
							We have
							$
							L 
							=
							\frac{1}{4}
							\lambda_{\max}\left(
							\frac{A^\top A}{n_{\rm tr}}
							\right).
							$
							The majorant parameter is then set to
							$
							\widehat L
							=
							\max\{10^{-6},\,1.1 L\}.
							$
							
							Now, we describe how to evaluate the robust test criteria in
							\eqref{eq:rlr-robust-loss}--\eqref{eq:rlr-cvar-loss}. 
							
							Let
							$
							\mathcal E_r
							:=
							\{-r,r\}^m
							$
							denote the vertex set of \(\cY_r\). For any fixed classifier parameter
							\(x\), the perturbed test feature \(\widetilde u_i^{\rm te}(y)\) is affine in
							\(y\). Therefore, the test loss
							$
							\ell_i^{\rm te}(x,y)
							=
							\log\left(
							1+\exp\bigl(
							-v_i^{\rm te}(\widetilde u_i^{\rm te}(y))^\top x
							\bigr)
							\right)
							$
							is convex in \(y\), because it is the logistic loss composed with an affine
							mapping. Consequently, the average test loss
							$
							y\mapsto
							\frac{1}{n_{\rm te}}
							\sum_{i=1}^{n_{\rm te}}
							\ell_i^{\rm te}(x,y)
							$
							is also convex in \(y\). Since a convex function over a compact polytope attains
							its maximum at an extreme point, the robust test loss in
							\eqref{eq:rlr-robust-loss} is computed exactly as
							\[
							L_{\rm rob}(x)
							=
							\max_{y\in\mathcal E_r}
							\frac{1}{n_{\rm te}}
							\sum_{i=1}^{n_{\rm te}}
							\ell_i^{\rm te}(x,y).
							\]
							In all robust logistic regression experiments, \(m=5\). Hence this exact
							evaluation requires only \(2^5=32\) vertex evaluations.
							
							The samplewise worst-case losses used in the CVaR criterion are computed in the
							same way. Specifically, for each test sample \(i\), we evaluate
							\[
							L_i^{\rm wc}(x)
							=
							\max_{y\in\mathcal E_r}
							\ell_i^{\rm te}(x,y),
							\qquad i=1,\ldots,n_{\rm te}.
							\]
							After obtaining the samplewise worst-case losses
							\(\{L_i^{\rm wc}(x)\}_{i=1}^{n_{\rm te}}\), we sort them as
							\[
							L_{(1)}^{\rm wc}(x)
							\le
							L_{(2)}^{\rm wc}(x)
							\le
							\cdots
							\le
							L_{(n_{\rm te})}^{\rm wc}(x).
							\]
							The empirical \(\operatorname{CVaR}_{\alpha}\) value in
							\eqref{eq:rlr-cvar-loss} is then evaluated from this sorted list. Let
							\[
							k_\alpha:=\lceil \alpha n_{\rm te}\rceil .
							\]
							An equivalent finite-sample expression for \eqref{eq:rlr-cvar-loss} is
							\[
							\operatorname{CVaR}_{\alpha}(x)
							=
							\frac{
								\bigl(k_\alpha-\alpha n_{\rm te}\bigr)
								L_{(k_\alpha)}^{\rm wc}(x)
								+
								\sum_{i=k_\alpha+1}^{n_{\rm te}}
								L_{(i)}^{\rm wc}(x)
							}{
								(1-\alpha)n_{\rm te}
							}.
							\] 
							
							%
							%
							%
							
							
						\end{APPENDICES}

					\end{document}